%% file: polynomial_diameter10.tex
\documentclass{amsart}[12pt]

\usepackage{array}
\usepackage{bm}
\usepackage{subfig, caption}

\usepackage{tikz}
	\tikzstyle{vertex} = [circle, draw, fill=black!50, inner sep=0pt, minimum width=4pt]
	\tikzstyle{state} = [circle, draw]
	\tikzstyle{transition} = [graph edge, /eq={shape = ellipse, inner sep = 2pt}, font =]
	\tikzstyle{label} = [ellipse , inner sep = 1pt]
	\tikzstyle{every loop} = [min distance = .5cm]
	\tikzstyle{graph edge} = [thick,font = \small, ->, every node/.style = {label}]
	\tikzstyle{a} = ["$a$", graph edge, thin]
	\tikzstyle{b} = ["$b$", graph edge, very thick]
	\tikzstyle{c} = ["$c$", graph edge, densely dotted]
	\pgfkeys{/eq/.style={every edge quotes/.add style = {}{#1}}}
	\pgfkeys{/loop at/.style={every loop, out = #1-15, in = #1+15}}
	\pgfkeys{/wide loop at/.style={every loop, min distance = 1cm, out = #1-30, in = #1+30}}
	\pgfkeys{/dir bend/.style 2 args = {out = #1 + #2, in = 180 + #1 - #2}}
	\pgfkeys{/from to/.style n args = {3}{"$#1 \ft #2$" {#3}}}
\usetikzlibrary{matrix,quotes,math,shapes}

\include{preamble}

\makeatletter

\newlength{\vdropheight}
\newcommand{\vdrop}[2][\vdropheight]{
	\raisebox{\minof{#1-\height}{0pt}}{#2}
}

\newcommand{\edgetol}{\@ifstar{\edgetol@drop}{\edgetol@nodrop}}

\newlength{\edgetol@width}
\def\edgetolfont{\small}
\newcommand{\edgetol@nodrop}[2][]{
	\settowidth{\edgetol@width}{\edgetolfont $#2$}
	\pgfmathsetlength{\edgetol@width}{max(\edgetol@width+4pt,1.5em)}
	\mathrel{%
		\tikz[baseline=-\the\dimexpr\fontdimen22\textfont2\relax]{
			\draw (0,0) edge[graph edge, /eq = {rectangle, inner sep = 2pt,font = \edgetolfont}, #1, "\raisebox{3pt}{$#2$}" {inner sep = 0pt}] (\the\edgetol@width,0);
		}%
	}
}

\newlength{\edgetol@dropheight}
\newcommand{\edgetol@drop}[2][]{
	\settoheight{\edgetol@dropheight}{$a^b$}
	\mathrel{%
		\vdrop[\edgetol@dropheight]{$\edgetol@nodrop[#1]{#2}$}%
	}
}

\makeatother

\newcommand{\edgetoa}{\edgetol[a]{}}
\newcommand{\edgetob}{\edgetol[b]{}}
\newcommand{\edgetoc}{\edgetol[c]{}}

\newcommand{\act}[2]{\prescript{#1\!}{}{#2}}

\newcommand{\T}{\mathbf T}
\renewcommand{\S}{\mathbf S}

\newcommand{\words}[1]{#1^{*,\infty}}
\newcommand{\rwords}[1]{#1^{-\infty,*}}

\DeclareMathOperator{\aAut}{FAut}

\newcommand{\sigmaw}{\bm{\sigma}}
\newcommand{\tauw}{\bm{\tau}}

\renewcommand{\sec}[2]{{#1|}_{#2}}

\newcommand{\ft}{\mathop:}

\begin{document}

\title{Lifts, derandomization, and diameters of Schreier graphs \\
of Mealy automata}

\author[Anton~Malyshev]{ \ Anton~Malyshev$^\star$}
\author[Igor~Pak]{ \ Igor~Pak$^\star$}

\thanks{\thinspace ${\hspace{-.45ex}}^\star$Department of Mathematics,
UCLA, Los Angeles, CA, 90095.
\hskip.06cm
Email:
\hskip.06cm
\texttt{\{amalyshev,pak\}@math.ucla.edu}}

\begin{abstract}
It is known that random \emph{$2$-lifts} of graphs give rise to expander graphs. We present a new conjectured derandomization of this construction based on certain \emph{Mealy automata}. We verify that these graphs have polylogarithmic diameter, and present a class of automata for which the same is true.  However, we also show that some automata in this class do not give rise to expander graphs.
\end{abstract}

\maketitle
\theoremstyle{plain}

\section{Introduction}\label{section:introduction}

\noindent
In~\cite{random_lifts_expansion}, Bilu and Linial showed that random \emph{$2$-lifts} of expanding graphs remain expanding with high probability. This gives a probabilistic construction of expander families.  Several ways to derandomize this procedure are also given in~\cite{random_lifts_expansion}, but none of them give a \emph{strongly explicit} description of a family of expander graphs.  That is, a description in which the actual graph is much larger than working memory, but a computer can list neighbors of a vertex in polylogarithmic (in the size of the graph) time.

We consider the following two families of $2$-lifts of graphs.  The \emph{Aleshin graphs} $A_0, A_1, A_2, \dots$ are a sequence of $3$-regular edge-labeled directed graphs. The first graph $A_0$ is defined to be a single vertex with three self-loops labeled $a$, $b$, and $c$. Given the graph $A_n$, the next graph $A_{n+1}$ is defined as a certain \emph{graph lift} of $A_n$: Each vertex $v \in A_n$ lifts to two vertices $v_0, v_1 \in A_{n+1}$, and the edges transform as follows:
\tikzstyle{cmatrix} = [matrix of math nodes, ampersand replacement=\&, column sep = 1.6em, row sep = 4pt, inner sep = 2pt, nodes={anchor = base west}]
\def\crossangle{15}
\newcommand{\cross}[6]{
	\tikz[baseline, every edge quotes/.style = {rectangle, pos = .05,auto}]{
		\matrix (m)
		[cmatrix]{
			#1 \& #2\\
			#3 \& #4\\
		};
		\draw
			(m-1-1.east) edge[#5,out = 0-\crossangle, in = 180-\crossangle] (m-2-2.west)
			(m-2-1.east) edge[#6,out = 0+\crossangle, in = 180+\crossangle, swap] (m-1-2.west)
		;
	}
}
\newcommand{\nocross}[6]{
	\tikz[baseline]{
		\matrix (m)
		[cmatrix]{
			#1 \& #2\\
			#3 \& #4\\
		};
		\draw
			(m-1-1.east) edge[#5] (m-1-2.west)
			(m-2-1.east) edge[#6] (m-2-2.west)
		;
	}
}
\newcommand{\crossvw}[2]{
	\cross{v_0}{w_0}{v_1}{w_1}{#1}{#2}
}
\newcommand{\nocrossvw}[2]{
	\nocross{v_0}{w_0}{v_1}{w_1}{#1}{#2}
}
\begin{align*}
	v \edgetoa w
	\qquad \text{lifts to} & \qquad
	\nocrossvw{c}{c}
\\
	v \edgetob w \qquad \text{lifts to} & \qquad
	\crossvw{a}{b}
\\
	v \edgetoc w \qquad \text{lifts to} & \qquad
	\cross{v_0}{w_0}{v_1}{w_1.}{b}{a}
\end{align*}
That is, e.g., if $A_n$ has an edge labeled $c$ from $v$ to $w$, then $A_{n+1}$ has an edge labeled $b$ from $v_0$ to $w_1$, and an edge labeled $a$ from $v_1$ to $w_0$.

Another family, the \emph{Bellaterra graphs} $B_0, B_1, B_2, \dots$ is defined the same way, except with transformation rules
\begin{align*}
	v \edgetoa w \qquad \text{lifts to} & \qquad
	\crossvw{c}{c}
\\
	v \edgetob w \qquad \text{lifts to} & \qquad
	\nocrossvw{a}{b}
\\
	v \edgetoc w \qquad \text{lifts to} & \qquad
	\nocross{v_0}{w_0}{v_1}{w_1.}{b}{a}
\end{align*}
It is not hard to check that the reverse of every edge in $B_n$ is also in $B_n$, so these can be thought of as undirected graphs. The first few graphs in this family are pictured in Figure~\ref{figure:bellaterra_graphs}.

\begin{figure}
\input{bellaterra_graphs.tex}
\label{figure:bellaterra_graphs}
\end{figure}

The main result of this paper is the following theorem:

\begin{thm}
The diameter of the Aleshin graphs $\{A_i\}_{i=1}^\infty$ and Bellaterra graphs $\{B_i\}_{i=1}^\infty$ grows at most quadratically in $n$, i.e.,
$$\diam(A_n) = O(n^2) \quad\text{and}\quad \diam(B_n) = O(n^2) \quad \text{as} \ \ \, n\to \infty.$$
\end{thm}

Prior to this paper, there were no nontrivial bounds on the diameter of~$A_n$; even  subexponential bounds remained out of reach. Note also that in principle we can start with \emph{any} 3-labeled graph in place of $A_0=B_0$, and proceed making lifts as above.  We do not consider these in the paper, and our algebraic techniques do not apply.

\smallskip

 Observe that both families of graphs are \emph{very explicit} in the following sense: there is a polynomial time algorithm which, given a number $n$ and $v \in \Gamma_n$, lists the neighbors of $n$. ``Polynomial time'' here refers to a runtime which is polynomial in the number of bits necessary to describe the input. It takes $n$ bits to describe a vertex of $B_n$ or $A_n$, so the algorithm should run in time $O(n^d)$, for some~$d$.

In particular, it follows that they are \emph{strongly explicit} in the sense of \cite{random_lifts_expansion}: There is a polynomial time (in the size of the inputs) algorithm which, given a number $n$, and vertices $v, w \in \Gamma_n$, decides whether $v$ and $w$ are adjacent in $\Gamma_n$.

As we will see below, these graphs can be described in terms of invertible \emph{Mealy automata}. The associated automata are small: they act on binary strings and have only 3 states. A detailed study of all such small automata was performed in~\cite{3_state_automata}. The Bellaterra and Aleshin automata are numbered $846$ and $2240$ in that article. They are the only nontrivial \emph{bireversible} ones. Spectra of the first few associated graphs are also computed in~\cite{3_state_automata}, and the data suggest that the Aleshin graphs are a family of expanders with eigenvalue gap roughly~$0.2$.

\begin{conj}\label{conj:aleshin_expanders}
The Aleshin graphs $\{A_i\}_{i=1}^\infty$ are a family of two-sided expanders.
\end{conj}

Here by \emph{two-sided} we mean that both the second largest and the
smallest eigenvalues of $3$-regular graphs $A_n$ are bounded away
 as follows: $\lambda_2<3-\varepsilon$ and $\lambda_n > -3+\ve$ (see e.g.~\cite{Tao}).

\begin{figure}
\begin{tikzpicture}[x = .5cm , y = 10cm]
	
    \pgfmathsetmacro\bbleft{1}
    \pgfmathsetmacro\bbright{20.5}
    \pgfmathsetmacro\bbbottom{0}
    \pgfmathsetmacro\bbtop{.37}

	\draw [very thin,color=black!10!white, xstep = 1, ystep = .1]
		(\bbleft , \bbbottom) grid (\bbright,\bbtop)
	;

	\foreach \x  in {1,...,20}{
		\pgfmathsetmacro\xtext{ \x }
		\draw (\x ,2pt) -- (\x ,-2pt) node[below] {\tiny $\xtext$};
	}
	
	\draw [->] (\bbleft, \bbbottom) -- (\bbright, \bbbottom);

    \foreach \y  in {0,...,3}{
    	\pgfmathparse{\y*.1}
        \pgfmathsetmacro\ytext{\pgfmathresult}
        \draw (\bbleft, \ytext) +(-2pt,0) -- +(2pt,0) node[ellipse, left] {\tiny $\pgfmathprintnumber{\ytext}$};
    }
    \draw[->] (\bbleft, \bbbottom) -> (\bbleft, \bbtop);

    \begin{pgfinterruptboundingbox}
	\path[clip] (0cm, \bbbottom) rectangle (100cm, \bbtop);
	
	\draw plot file{aleshin_gaps.dat};
	\draw (7, 0.258) node [above right, font = \small]{Aleshin};
	\draw[thick, densely dotted] plot file{bellaterra_gaps.dat};
	\draw (6, 0.067) node[above right, font = \small]{Bellaterra};
	\end{pgfinterruptboundingbox}

\end{tikzpicture}

\caption{Eigenvalue gaps of the Bellaterra and Aleshin graphs.}
\label{figure:ab_gaps}
\end{figure}
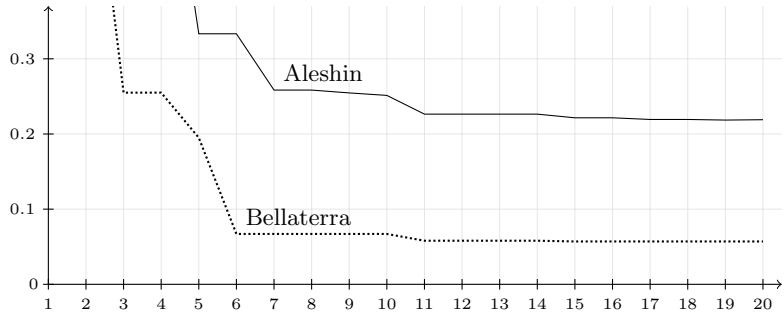

Though it is less clear from the data in~\cite{3_state_automata}, our own computations (see Figure~\ref{figure:ab_gaps}) suggest that the Bellaterra graphs are also expanders, with eigenvalue gap roughly $0.05$, so we make the stronger conjecture:\footnote{See Remark~\ref{remark:bellaterra_implies_aleshin}.}

\begin{conj}\label{conj:bellaterra_expanders}
The Bellaterra graphs $\{B_i\}_{i=1}^\infty$ are a family of two-sided expanders.
\end{conj}

If so, they are a strongly explicit derandomization of the probabilistic construction in~\cite{random_lifts_expansion}. One consequence of being an expander family is logarithmic diameter growth with respect to the size of the graph, so if Conjecture~\ref{conj:aleshin_expanders} holds then $\diam(A_n)$ grows linearly in~$n$,  stronger claim than in the theorem.

Unfortunately, we are not near proving either conjectures and in fact our tools are too weak to prove them.  Later in the paper, we state and prove general conditions on an automaton which guarantee polynomial diameter growth in the associated graphs (Section~\ref{section:general}).  We then prove that for some automata which satisfy those conditions, we do not get expanders (see Section~\ref{section:examples}).  In other words, a different, perhaps combinatorial technique is needed to prove the expansion.


\section{Mealy automata}\label{section:mealy_automata}

The Bellaterra graphs $\{B_n\}_{n=1}^\infty$ are \emph{very explicit} in the sense of~\cite{expander_survey}.\footnote{Sometimes, these are called \emph{fully explicit}, see e.g.~\cite{Vad}.} That is, there is a polynomial time algorithm which, given a number $n$ and a vertex $v \in B_n$, lists the neighbors of~$v$ in~$B_n$. It takes $n$ bits to describe a vertex in $B_n$, so the runtime of the algorithm should be polynomial in~$n$.

In fact, there is a linear time algorithm. Even more strongly, the computation can be implemented with a Mealy automaton, i.e., a finite state automaton which outputs a letter each time it reads a letter.

\begin{Defn}
A \emph{Mealy automaton} $\cm = (Q,A,\tau,\sigma)$ is a pair of finite sets $Q$, $A$, together with functions $\sigma:Q \times A \ra A$, and $\tau:Q \times A \ra Q$.
\end{Defn}
The sets $Q$ and $A$ are called the \emph{states} and \emph{alphabet}, respectively. The functions $\sigma$ and $\tau$ are called the \emph{output} and \emph{transition} functions, respectively. When $\abs{Q} = q$ and $\abs{A} = a$, we call $\cm$ a $(q,a)$-automaton. We adopt the following notations:
\begin{align*}
\act{q}{x} = \sigma_q(x) = \sigma(q,x) \\
q ^x = \tau_x(q) = \tau(q,x).
\end{align*}

Let $A^*$ and $A^\infty$ denote the set of finite and infinite words in the alphabet $A$, respectively, and let $\words{A} = A^* \cup A^\infty$ denote the set of all words in $A$. A Mealy automaton in the state $q \in Q$ acts in a length-preserving way on words in $\words{A}$ by reading the first letter $x$, outputting the letter $\sigma(q,x)$, and acting on the rest of the word from the state $\tau(q,x)$. That is, each $q \in Q$ has a corresponding length-preserving map $\words{A} \ra \words{A}$ defined recursively by
\begin{align*}
\act{q}{(x_0 x_1 \dots x_n)} &= y_0 \act{r}{(x_1 \dots x_n)}, \\
\and
\act{q}{(x_0 x_1 x_2 \dots)} &= y_0 \act{r}{(x_1 x_2 \dots)},
\end{align*}
where $y_0 = \sigma(q,x_0)$ and $r = \tau(q,x_0)$. This extends to a left action of finite words $Q^*$ on words in $\words A$ via, e.g.,
\begin{align*}
\act{qr}{s} = \act{q}{(\act{r}{s})}.
\end{align*}

So we've defined an extension of $\sigma: Q \times A \ra A$ to a map $\sigmaw: Q^* \times \words A \ra \words A$ given by
\begin{align*}
\sigmaw(w,s) = \sigmaw_w(s) &= \act{w}{s}.
\end{align*}

A Mealy automaton can be depicted with a Moore diagram: a directed graph with a vertex for each state $q \in Q$ and a labeled edge
$$
q \edgetol{x \ft y} r
$$
for every $q \in Q$ and every $x \in A$, where $y = \sigma(q,x)$ and $r = \tau(q,x)$. That is, an edge $q \edgetol{x \ft y} r$ denotes that if the Mealy automaton is in state $q$ and reads the letter $x$, then it outputs the letter $y$ and transitions to the state $r$. We will sometimes simply write
$q \edgetol*{x \ft y} r$ to denote that $y = \sigma(q,x)$ and $r = \tau(q,x)$.

\begin{figure}
\input{bellaterra_automaton}
\label{figure:bellaterra_automaton}
\end{figure}

\begin{Example}\label{example:bellaterra_automaton}
Consider the \emph{Bellaterra automaton} $\cb$ pictured in Figure~\ref{figure:bellaterra_automaton}. More formally, $\cb = (Q, A, \tau, \sigma)$ is defined by
\begin{align*}
&A = \{0, 1\}, \quad Q = \{a,b,c\}
\\
&\sigma_{a} = \sigma_b = (0) (1), \quad \sigma_c = (0~1),
\\
\text{and } &\tau_0 = (a~b~c), \quad \tau_1 = (a~c) (b),
\end{align*}
where we use the usual cycle notation for permutations, so e.g., $\tau_0(a) = b$, $\tau_0(b) = c$, $\tau_0(c) = a$.

Then given a number $n$, the Bellaterra graph $B_n$ can be described as the graph whose vertices are length $n$ binary strings, with an edge
\begin{align*}
s \edgetol{q} \left(\act{q}{s}\right)
\end{align*}
for each vertex $s \in A^n$ and each state $q \in Q$. For example, we have
\begin{gather*}
\act{c}{(0000)}
= 1\act{a}{(000)}
= 10\act{b}{(00)}
= 100\act{c}{(0)}
= 1001,
\\
\text{so}\quad
0000 \edgetoc 1001.
\end{gather*}
\end{Example}

Some symmetry between states and letters of a Mealy automaton is already apparent in the definition. The nature of this symmetry becomes more clear if we consider computing compositions of maps associated to the states of an automaton, we have, e.g.,
\begin{align*}
\act{q_1q_0}(x_0 x_1 \dots x_n)
= \act{q_1} (y_0 \act{r_0} (x_1 \dots x_n))
= z_0 \act{r_1 r_0} (x_1 \dots x_n) = \dots,
\end{align*}
where $q_0 \edgetol*{x_0 \ft y_0} r_0$, and $q_1 \edgetol*{y_0 \ft z_0} r_1$. 
The computation proceeds by taking any instance of $\act{q}(x \dots)$ in the expression, and replacing it with $y \act{r}{(\dots)}$, where $q \edgetol*{x \ft y} r$.

If we ignore parentheses, states in $Q$ and letters in $A$ play a symmetric role in this process, except that letters in $Q$ are written higher and disappear when they are at the right side of the expression. Taking this symmetry into account, the automaton also naturally defines an action of the letters in $A$ on finite words in $Q^*$:
\begin{align*}
(q_n \dots q_1 q_0)^x = (q_n \dots q_1)^y r_0,
\end{align*}
where $q_0 \edgetol*{x \ft y} r_0$. Letters in $A$ also act on the set of \emph{left}-infinite words in the alphabet $Q$:
\begin{align*}
(\dots q_2 q_1 q_0)^x = (\dots q_2 q_1)^y r_0.
\end{align*}
We let $Q^{-\infty}$ denote this set of left-infinite words, and let $\rwords Q$ denote $Q^* \cup Q^{-\infty}$, so we have an action of $A$ on $\rwords Q$.
This naturally extends to a \emph{right} action of $A^*$ on $\rwords Q$, via, e.g.
$$
w^{xy} = (w^x)^y.
$$
So we have defined a map $\tauw: \rwords Q \times A^* \ra \rwords Q$, given by
\begin{align*}
\tauw(w,s) = \tauw_s(w) &= {w}^{s}.
\end{align*}
It is straightforward to check that for any $s \in A^*, t \in \words A, w \in Q^*, v \in \rwords Q$, the actions we have defined satisfy the following relations:
\begin{align*}
\act{w}{(st)} &= \wt s \, \act{\wt w}{(t)},
\rlap\and \\
(vw)^s &= (v)^{\wt s} \, \wt w,
\end{align*}
where $\wt s = \act{w}{s}$ and $\wt w = w^s$.

If we need to specify the automaton, we will write $\sigmaw_\cm = \sigmaw$ and $\sigmaw_{\cm, w} = \sigmaw_w$, and similarly for $\tauw$.

With this symmetry in mind, it is sensible to define the \emph{dual} of an automaton $\cm = (Q,A,\tau,\sigma)$ to be the automaton $\ol \cm = (\wh Q, \wh A, \wh \tau, \wh \sigma)$ given by interchanging the roles of the states and alphabet. That is, we take
\begin{align*}
\wh A = Q, \quad \wh Q = A, \quad \wh \sigma(a,q) = \tau(q,a), \quad \text{and}  \quad \wh \tau(a,q) = \sigma(q,a).
\end{align*}
In other words, for $q, r \in Q$ and $x,y \in A$, we have $x \edgetol*{q \ft r} y$ in $\ol \cm$ if and only if $q \edgetol*{x \ft y} r$ in $\cm$.

Computations in the dual automaton are computations in the original automaton, with each step written backwards. It follows that, e.g., for every $s \in A^*$ and $w \in Q^*$ we have
$$
{\sigmaw_{\ol \cm}(s, w)} = \ol{\tauw_{\cm}(\ol w,\ol s)},
$$
where $\ol u$ denotes the reversal of $u$.

\begin{Example}
The dual of the Bellaterra automaton is also pictured in Figure~\ref{figure:bellaterra_automaton}.
\end{Example}

\begin{Example}
Let $A = \{0, 1\}$. Consider the Mealy automaton pictured in Figure~\ref{figure:adding_automaton}.
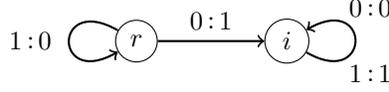
\begin{figure}
\begin{tikzpicture}[]

	\node at (0,0) [state] (r) {$r$};
	\node at (2,0) [state] (i) {$i$};
	
	\draw
		(r) edge[transition, /wide loop at = 180, /from to = {1}{0}{left}] (r)
		(r) edge[transition, /from to = {0}{1}{above},] (i)
		(i) edge[transition, /wide loop at = 0, /from to = {0}{0}{near end, above right}, /from to = {1}{1}{near start, below right}] (i)
	;
\end{tikzpicture}
\caption{The binary adding machine}
\label{figure:adding_automaton}
\end{figure}
The map $\sigmaw_r:\{0,1\}^* \ra \{0,1\}^*$ is simply addition of $1$, where length $n$ words in $A^*$ are interpreted as binary representations of numbers modulo $2^n$, with the least significant digit on the left.
\end{Example}

We say a Mealy automaton is \emph{invertible} if $\sigma_q$ is invertible for every $q \in Q$. This occurs if and only if the endomorphism $\sigmaw_w:A^* \ra A^*$ is invertible for every $w \in Q^*$. We are primarily interested in invertible automata, though our results can be generalized to the non-invertible case.

The \emph{inverse} of an invertible automaton $\cm = (Q,A,\tau,\sigma)$ is the automaton $\cm^{-1} = (Q',a,\tau',\sigma')$ given by
\begin{align*}
Q = \setc{q'}{q \in Q}, \quad \sigma'_{q'} = \sigma_q^{-1}, \quad \tau'(q',a) = \tau(q,\sigma_{q}^{-1}(a)).
\end{align*}
It is straightforward to check that $\sigmaw_{\cm^{-1}, q'} = \sigmaw_{\cm, q}^{-1}$ for every $q \in Q$.

Consider two automata $\cm = (Q,A,\tau,\sigma)$, $\cm' = (Q',A,\tau',\sigma')$ acting on the same alphabet, with $Q, Q'$ disjoint. Their \emph{union} is the automaton $\cm \cup \cm' = (Q \cup Q', A, \tau'',\sigma'')$, where
\begin{gather*}
\tau'' (q,a) = \begin{cases}
\tau(q,a) & q \in Q \\
\tau'(q,a) & q \in Q' \\
\end{cases}
\qquad
\and
\qquad
\sigma'' (q,a) = \begin{cases}
\sigma(q,a) & q \in Q \\
\sigma'(q,a) & q \in Q' \\
\end{cases}
\end{gather*}
For example, $\cm \cup \cm^{-1}$ is an automaton with twice as many states as $Q$, in which every state $q$ has an inverse state $q'$ with $\sigmaw_{q'} = \sigmaw_q^{-1}$.

We say an automaton is \emph{reversible} if its dual is invertible.

We say an automaton is \emph{bireversible} if it is invertible, reversible, and its inverse is reversible. Note that the last condition does not follow from the other two. For example, the three-state automaton in Figure~\ref{figure:23_automaton} is reversible and invertible, but not bireversible.


\section{Schreier graphs}\label{section:schreier_graphs}

For our purposes, graphs are locally finite, directed, and may have self-loops and repeated edges. A graph is \emph{regular} if the indegree and outdegree are the same across all vertices.

Let $\Gamma$ be a graph. Given vertices $v, w \in \Gamma$, we write $v \edgeto_\Gamma w$ if there is an edge in $\Gamma$ from $v$ to $w$. We write $d_\Gamma(v,w)$ for the distance between $v$ and $w$, i.e.~the length of the shortest undirected path between $v$ and $w$. When there is no such path, we take $d_\Gamma(v,w) = \infty$. Given a nonnegative integer $r$, the \emph{ball} of radius $r$ centered at $v$ is the set
$$
B_\Gamma(v,r) = \{w \in \Gamma : d(v,w) \leq r\}.
$$
The \emph{diameter} of $\Gamma$ is defined to be
$$
\diam(\Gamma) = \max_{v,w \in \Gamma} d_\Gamma(v,w).
$$
When it is clear from context what graph we are discussing, we will drop the subscripts and simply write $v \edgeto w$, $d(v,w)$, and $B(v,r)$.

In Example~\ref{example:bellaterra_automaton} we described the Bellaterra graphs in terms of a Mealy automaton. In the same way, we can associate a sequence of graphs to any Mealy automaton. Since we are primarily concerned with regular graphs, we require the automaton to be invertible.

\begin{Defn}
Let $\cm = (Q, A, \sigma, \tau)$ be an invertible Mealy automaton. Given $n \in \{1, 2, \dots\} \cup \{\infty\}$, the \emph{Schreier graph} $\Gamma_{\cm, n}$ is a directed graph, defined as follows: The vertices of $\Gamma_{\cm, n}$ are length $n$ words in $\words A$, i.e.~elements of $A^n$. For each vertex $s \in \Gamma_{\cm, n}$ and each state $q \in Q$, the Schreier graph $\Gamma_{\cm, n}$ has an edge $$s \edgeto \act{q}{s}.$$
\end{Defn}

Clearly, the number of edges leaving a vertex is $\abs{Q}$. The Schreier graph of the inverse automaton, $\Gamma_{\cm^{-1},n}$, is simply $\Gamma_{\cm,n}$ with the edges reversed. So, the number of edges entering a given vertex in $\Gamma_{\cm^{-1},n}$ is also $\abs{Q}$, and $\Gamma_{\cm,n}$ is regular.

\begin{Example}
The $n$-th Bellaterra graph $B_n$ is the Schreier graph $\Gamma_{\cb,n}$, where $\cb$ is the Bellaterra automaton, pictured in Figure~\ref{figure:bellaterra_automaton}.
\end{Example}

\begin{Example}
The $n$-th Aleshin graph $A_n$ is the Schreier graph $\Gamma_{\ca,n}$, where $\ca$ is the Aleshin automaton, first considered in \cite{aleshin}, pictured in Figure~\ref{figure:aleshin_automaton}.
\end{Example}

\begin{figure}
\input{aleshin_automaton}
\label{figure:aleshin_automaton}
\end{figure}


\section{Automaton groups}\label{section:automata_groups}

Let $\cm = (Q,A,\tau,\sigma)$ be an invertible Mealy automaton. As seen in Section~\ref{section:mealy_automata}, we have invertible maps $\sigmaw_q : \words A \ra \words A$ for each $q \in Q$. This gives an action of the free group $F_Q$ on $\words A$. We can extend the definition of $\sigmaw$ as follows: For $w \in F_Q$, we can define $\sigmaw_w$ in the natural way, e.g.,
$$\sigmaw_{qr^{-1}} = \sigmaw_{q} \, \sigmaw_{r}^{-1}.$$
As usual, we will adopt the notational convention
$$
\sigmaw_w(s) = \sigmaw(w,s) = \act{w}{s}.
$$

The \emph{automaton group} associated to $\cm$ is the group $G_\cm$ generated by the automorphisms $\sigmaw_q$.

For example, letting $\ca$ denote the Aleshin automaton, it was shown in \cite{VV} that $\sigmaw_{\ca, a}$, $\sigmaw_{\ca, b}$, and $\sigmaw_{\ca, c}$ satisfy no nontrivial relation, so $G_\ca$ is the free group $F_3$. However, it is straightforward to check that $\sigmaw_{\cb, a} ^ 2 = \sigmaw_{\cb,b} ^ 2 = \sigmaw_{\cb,c} ^ 2 = \id$, where $\cb$ is the Bellaterra automaton. It can be shown (see, e.g., \cite{Nek, 3_state_automata}) that the words satisfy no other relation, so we say $G_\cb \cong \gen{a, b, c \mid a^2, b^2, c^2} = C_2 * C_2 * C_2$.

Information about $G_\cm$ as an abstract group can be used to obtain information about the Schreier graphs $\Gamma_{\cm,n}$. See Remarks~\ref{remark:property_t_expanders}~and~\ref{remark:amenable_nonexpanders}.


\section{Trees and automorphisms}\label{section:trees}

In this context it is natural to think of the set of finite words $A^*$ as vertices in a regular rooted tree, where the empty word is the root and the children of the word $s$ are the words $sx$ for $x \in A$. We will need to talk about rooted trees more generally, so we make the following definitions.

\begin{Defn}
A \emph{rooted tree} (or simply \emph{tree}) is a graph $\T$ with a distinguished root vertex $r \in \T$ such that for each $v \in \T$ there is exactly one directed path from $r$ to $v$. The \emph{level} of $v$, denoted $\len(v)$ is the length of this path. The \emph{$n$-th level} of $\T$, denoted $\T_n$ is the set of all vertices $v \in \T$ such that $\len(v) = n$. A \emph{subtree} of $\T$ is a subgraph containing $r$ which is itself a rooted tree. A \emph{tree isomorphism} between two trees $\S$ and $\T$ is a graph isomorphism which sends the root of $\S$ to the root of $\T$. An \emph{automorphism} of $\T$ is an isomorphism from $\T$ to $\T$. The automorphisms of $\T$ form a group, and we denote it $\Aut(\T)$.
\end{Defn}

Then, given a Mealy automaton $\cm = (Q,A,\sigma,\tau)$, for any $q \in Q$ the map $\sigmaw_q : A^* \ra A_*$ is a tree automorphism. That is, $\sigmaw_q$ is a bijection which fixes the empty word, and sends children of $x$ to children of $\sigmaw_q(x)$. In other words, for every $s \in A_*$ and $x \in A$ there is some $y \in A$ such that
$$
\sigmaw_q(sx) = \sigmaw_q(s)y.
$$
Of course, it follows that for any $w \in Q^*$, the map $\sigmaw_w:A^* \ra A_*$ is a composition of tree automorphisms and is itself a tree automorphism.

Infinite words, i.e., elements of $A^\infty$, can be thought of as rays in the tree $A^*$, and $\sigmaw_w$ acts on them in the natural way.

Note that in order to think of $\tau_a: Q^* \ra Q^*$ as a tree automorphism, we must think of $Q^*$ as a tree in the reverse way, i.e.~the children of $w$ are of the form $qw$ for $q \in Q$, rather than of the form $wq$.

Given a tree automorphism $g: A^* \ra A^*$ and a word $s \in A^*$, the \emph{section} of $g$ at $s$, is the tree automorphism $\sec{g}{s}: A^* \ra A^*$ defined by
$$
g(st) = g(s) \sec{g}{s}(t).
$$
Note that we are using a canonical identification between branches of the tree $A^*$. There need not be such an identification in a general tree, so this definition of sections is specific to trees of words.

We call a tree automorphism $\alpha:A^* \ra A^*$ \emph{automatic} if it $\alpha = \sigmaw_{\cm, q}$ for state $q$ of some Mealy automaton $\cm$. Equivalently, $\alpha$ is automatic if and only if it has finitely many sections. The set of automatic automorphisms forms a subgroup $\aAut(A^*) < \Aut(A^*)$.

An automorphism $g: A^* \ra A^*$ is determined by its action on the first level, $(A^*)_1 = A^1 = A$, and its sections $\sec{g}{x}$ at all $x \in A$. If $\rho: A \ra A$ is a permutation, then for notational convenience we can extend $\rho$ as an automorphism $A^* \ra A^*$ via
$$
\rho(xs) = \rho(x)s.
$$
If $A$ is equipped with an ordering of its elements, say, $A = \{x_1, \dots, x_k\}$, then we write
$$
(g_1, \dots, g_k)
$$
for the automorphism $g:A_* \ra A_*$ which acts trivially on $A$, and for which $\sec{g}{x_i} = g_i$ for all $i$. Then every automorphism can be uniquely decomposed into
$$
g = \rho \, (g_1, \dots, g_k),
$$
for some permutation $\rho:A \ra A$ and some automorphisms $g_i: A^* \ra A^*$. Specifically, $\rho$ is the restriction of $g$ to $A$, and $g_i = \sec{g}{x_i}$. Then, given an invertible Mealy automaton $\cm = (Q,A,\tau,\sigma)$, the definition of the automorphisms $\sigmaw_q$ can be phrased recursively as
$$
\sigmaw_q = \sigma_q \, (\sigmaw_{q_1}, \dots, \sigmaw_{q_k}),
$$
where $q_i = q^{x_i}$. Such a recursive definition is called a \emph{wreath recursion.} For example, if $\cb = (Q,A,\tau,\sigma)$ is the Bellaterra automaton, then we have the wreath recursion
\begin{align*}
\sigmaw_a &= (\sigmaw_b, \sigmaw_c) \\
\sigmaw_b &= (\sigmaw_c, \sigmaw_b) \\
\sigmaw_c &= \rho (\sigmaw_a, \sigmaw_a),
\end{align*}
where $\rho : \{0,1\} \ra \{0,1\} $ swaps $0$ and $1$.

\begin{Defn}
Let $\T$ be a rooted tree. We say a tree automorphism $g:\T \ra \T$ is \emph{spherically transitive} (or just \emph{transitive}) if its restriction to every level of $\T$ is a transitive map.
\end{Defn}

For example, if $\cm$ is the adding automaton pictured in Figure~\ref{figure:adding_automaton}, then $\sigma_r:\{0,1\}^* \ra \{0,1\}^*$ is spherically transitive, because its action on the $n$-th level is addition of $1$ modulo~$2^n$.


\section{The Bellaterra automaton}\label{section:bellaterra}

Consider the Bellaterra automaton $\cb = (Q, A, \sigma, \tau)$, pictured in Figure~\ref{figure:bellaterra_automaton}. We want to show that the graphs $B_n = \Gamma_{\cb,n}$ have small diameter. Our approach is to find short words in $Q^*$ which change only the last digit of the word $1^n = 11\dots1$. So, we are looking for words which do not fix the infinite word $1^\infty = 111\dots$, but do preserve the first $n$ of its letters. It turns out there are enough of these words because $\tauw_1$ acts ``transitively enough'' on $Q^*$, so that almost every orbit under its action contains some word $w \in Q^*$ which swaps $0$ and $1$.

It is straightforward to check that $a^2$, $b^2$, and $c^2$ act trivially on $A^*$, (i.e.~$\sigmaw_{aa} = \sigmaw_{bb} = \sigmaw_{cc} = \id$) so we are primarily interested in \emph{reduced words} in $\{a,b,c\}$, i.e.~those which do not repeat the same letter twice in a row. Note that these words form a subtree of $Q^*$, which is nearly a binary tree: every vertex has two children, except the root.

We will need a simple result on the transitivity of automorphisms of a binary tree. $A = \{0,1\}$. Define a group homomorphism $\chi:\Aut(A^*) \ra \zz_2[[t]]$,  by
$$\chi(g) = \sum_{n=1}^\infty c_n t^{n-1},$$
where $(-1)^{c_n}$ is the sign of the permutation given by the action of $\chi$ the $n$-th level of $A^*$. Values of this homomorphism can be computed recursively via
$$
\chi(g) = c_1 + t\big(\chi(\sec{g}{x}) + \chi(\sec{g}{y})\big),
$$
where $c_1$ is $0$ if $g$ fixes the two elements of $A$, and $c_1 = 1$ if $g$ swaps them. We call $\chi(g)$ the \emph{characteristic function of $g$}. Of course, this definition makes sense when $A$ is any two-element set, so we will state the lemma more generally:

\begin{lemma}\label{lemma:binary_transitive}
Let $A = \{x,y\}$. An automorphism $g \in \Aut(A^*)$ is spherically transitive if and only if $\chi(g) = 1/(1-t)$.
\end{lemma}
\begin{proof}
If $g$ is spherically transitive, then its action on the $n$-th level of $A^n$ is a $(2^n)$-cycle, which is an odd permutation for all $n \geq 1$. Hence, $c_n = 1$ for all $n \geq 1$, and
$$
\chi(g) = \sum_{n=1}^\infty t^{n-1} = \frac 1 {1-t}\,.
$$

In the other direction, suppose
$\chi(g) = 1/(1-t)$, i.e.~$g$ acts as an odd permutation on $A^n$ for every $n \geq 1$. We will show by induction on $n$ that the action of $g$ on $A^n$ is a $(2^n)$-cycle for all $n \geq 0$. This is trivial for $n = 0$.

For the inductive step, suppose $g$ acts as a $(2^n)$-cycle on $A^n$. Given a word $s \in A^n$, we either have $g^{2^n}(sx) = sx$ or $g^{2^n}(sx) = sy$. In the first case $sx$ belongs to a $(2^n)$-cycle of $g$, in the second case, $sx$ belongs to a $(2^{n+1})$-cycle. So, any word in $A^{n+1}$ ending in $x$ belongs to either a $(2^n)$-cycle or a $(2^{n+1})$-cycle, and similarly for wards ending in $y$. So, the action of $g$ on $A^{n+1}$ decomposes into either two $(2^n)$-cycles or a single $(2^{n+1})$-cycle. But the former is an even permutation, so $g$ must act as a $(2^{n+1})$-cycle on $A^{n+1}$, as desired.
\end{proof}

\begin{lemma}\label{lemma:bellaterra_transitive}
Let $\cb = (Q,A,\tau,\sigma)$ denote the Bellaterra automaton. Then for every natural number $n$, the map $\tauw_{\cb, 1}$ acts transitively on the set of reduced words of length $n$ ending with $a$ or $c$.
\end{lemma}
\begin{proof}
We will write the argument down in terms of the dual automaton $\ol \cb = (A, Q,\ol \tau, \ol \sigma)$, pictured in Figure~\ref{figure:bellaterra_automaton}. Since taking the dual reverses words, we want to show that $\sigmaw_1 = \sigmaw_{\ol \cb, 1}$ acts transitively on the binary subtree $\T \subset Q^*$ of reduced words which \emph{begin} with $a$ or $c$.

It is convenient to put $\T$ into bijection with a binary tree of words $\mathbf{R} = \{\ua, \da\}^*$. We define the maps $\phi_a, \phi_b, \phi_c: \mathbf{R} \ra Q^*$ recursively by
\begin{align*}
\begin{tabular}{>{$}c<{$}@{$\qquad\qquad$}>{$}c<{$}}
\phi_a(\ua w) = b \, \phi_b(w) &
\phi_a(\da w) = c \, \phi_c(w) \\
\phi_b(\ua w) = c \, \phi_c(w) &
\phi_b(\da w) = a \, \phi_a(w) \\
\phi_c(\ua w) = a \, \phi_a(w) &
\phi_c(\da w) = b \, \phi_b(w) \\
\end{tabular}
\end{align*}

It is straightforward to check by induction on word length that for each $x \in Q$,  $\phi_x$ defines a tree isomorphism between $\mathbf{R}$ and the reduced words in $Q^*$ which do not begin with $x$. In particular, $\phi_b$ is a bijection between $\mathbf{R}$ and $\T$.

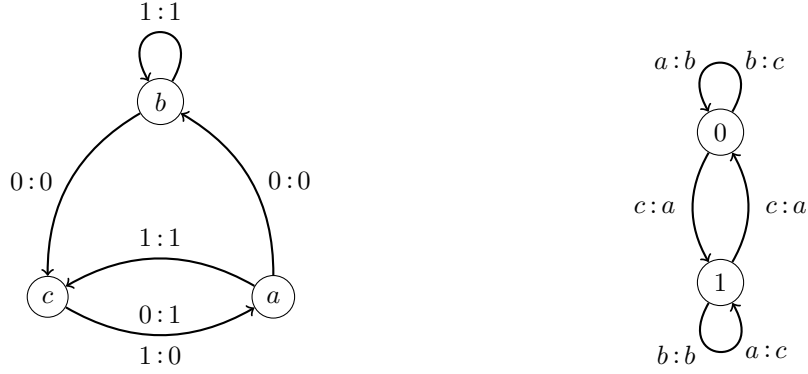
Now consider the dual $\ol \cb$ of the Bellaterra automaton, and in particular the corresponding automorphisms $\sigmaw_0, \sigmaw_1 \in \Aut(Q^*)$. Given $x, y \in Q$, $d \in A$, define
$$\sigmaw_{x,d,y} = \phi_x^{-1} \sigmaw_d \phi_y \in \Aut(\mathbf{R}).$$
Note that, a priori, the domain of $\phi_x^{-1}$ may not coincide with the image of $\sigmaw_d \phi_y$, so $\sigmaw_{x,d,y}$ may be ill-defined for some values of $x,d,y$. However, the computations below give an explicit recursion for computing $\sigmaw_{1,b,1}$, which also demonstrates that it is well-defined.

We can compute that, e.g.,
\begin{align*}
\sigmaw_{b,1,b}(\ua w)
&= \phi_b^{-1}(\sigmaw_1(\phi_b(\ua w)))\\
&= \phi_b^{-1}(\sigmaw_1(c \, \phi_c(w)))\\
&= \phi_b^{-1}(a \, \sigmaw_0(\phi_c(w)))\\
&= \da \phi_a^{-1}(\sigmaw_0(\phi_c(w)))\\
&= \da \sigmaw_{a,0,c}
\end{align*}

In particular, $\sec{\sigmaw_{b,1,b}}{\ua} = \sigmaw_{a,0,c}$.

Similar computations give the complete recursive description of $\sigmaw_{b,1,b}$, which we write down using the usual wreath recursion notation $g = \rho^\e (\sec{g}{\ua}, \sec{g}{\da})$, where $\rho$ swaps $\ua$ and $\da$:
\begin{align*}
\sigmaw_{b,1,b} &= \rho \, (\sigmaw_{a,0,c},\sigmaw_{c,1,a}) \\
\sigmaw_{a,0,c} &= (\sigmaw_{b,0,a},\sigmaw_{c,0,b}) \\
\sigmaw_{c,1,a} &= \rho \, (\sigmaw_{b,1,b}, \sigmaw_{a,0,c}) \\
\sigmaw_{b,0,a} &= (\sigmaw_{c,0,b}, \sigmaw_{a,1,c}) \\
\sigmaw_{c,0,b} &= (\sigmaw_{a,1,c}, \sigmaw_{b,0,a}) \\
\sigmaw_{a,1,c} &= \rho \, (\sigmaw_{c,1,a},\sigmaw_{b,1,b})
\end{align*}

Defining $F_{x,d,y} = \chi(\sigmaw_{x,d,y})$, this gives us the following linear equations in the ring $\zz_2[[t]]$:
\begin{align*}
F_{b,1,b} &= 1 + t(F_{a,0,c} + F_{c,1,a}) \\
F_{a,0,c} &= t(F_{b,0,a} + F_{c,0,b}) \\
F_{c,1,a} &= 1 + t(F_{b,1,b} + F_{a,0,c}) \\
F_{b,0,a} &= t(F_{c,0,b} + F_{a,1,c}) \\
F_{c,0,b} &= t(F_{a,1,c} + F_{b,0,a}) \\
F_{a,1,c} &= 1 + t(F_{c,1,a} + F_{b,1,b})
\end{align*}
Solving this system of equations yields
\begin{align*}
F_{b,1,b} &= 1/(1-t) \\
F_{a,0,c} &= 0\\
F_{c,1,a} &= 1/(1-t) \\
F_{b,0,a} &= t/(1-t) \\
F_{c,0,b} &= t/(1-t) \\
F_{a,1,c} &= 1,
\end{align*}
So we have$$\chi(\phi_b ^{-1} \sigmaw_{\ol \cb, 1} \phi_b) = 1/(1-t).$$

By Lemma~\ref{lemma:binary_transitive}, the automorphism $\phi_b ^{-1} \sigmaw_{ 1} \phi_b$ acts transitively on $\mathbf{R}$. Hence $\sigmaw_{1}$ acts transitively on $\T$, that is, for each $n$ it acts transitively on the set of length $n$ reduced words in $\{a,b,c\}$ which begin with $a$ or $c$. Hence, in the unreversed Bellaterra automaton $\cb$, we have that $\tauw_{\cb, 1}$ acts transitively on the words of a given length which end with $a$ or $c$.
\end{proof}

\begin{lemma}\label{lemma:orbit_nofix}
Let $\cm = (Q, A, \tau, \sigma)$ be a Mealy automaton, let $x$ be a letter in $A$, and let $w$ be a word in $Q^*$. Then $w$ stabilizes the infinite word $xxx\ldots = x^\infty$ if and only if every element of the orbit of $w$ under $\tauw_x$ stabilizes $x$. That is,
\begin{align*}
\act{w}(xxx\dots) = xxx\dots\qquad\text{if and only if} \qquad \sigmaw(\tauw_{x}^n(w), x) = x\text{ for all }n\geq 0.
\end{align*}
\end{lemma}
\begin{proof}
Say
\begin{align*}
\act{w}(xxx\dots) = y_0y_1y_2\dots.
\end{align*}
Then $y_n$ is the last letter of $\act{w}(x^{n+1})$. Letting $X = x^n$, we have
\begin{align*}
\act{w}(x^{n+1}) &
= \act{w}(Xx)
= \act{w}(X)\,\act{\tau(w, X)}x,
\end{align*}
So,
\begin{align*}
y_n &
= \act{\tau(w, X)}x
= \sigmaw(\tau(w, X),x)
= \sigmaw(\tau_X(w),x)
= \sigmaw(\tau_x^n(w),x),
\end{align*}
and therefore $\act{w}{xxx\dots} = xxx\dots$ if and only if $\sigmaw(\tauw_{x}^n(w), x) = x$ for every $n \geq 0$, as desired.
\end{proof}

\begin{thm}\label{theorem:bellaterra_diameter}
Let $B_n$ denote the $n$-th Bellaterra graph. Then $\diam(B_n) = O(n^2)$.
\end{thm}
\begin{proof}
Let $\cb = (Q,A, \tau, \sigma)$ denote the Bellaterra automaton, so that $B_n = \Gamma_{\cb,n}$. It is enough to show that for some $C$ the ball of radius $Cn^2$ around the vertex $1^n = 11\dots 1$ covers all of $B_n$. That is, we will show that for every number $n$, and every $v \in \cb_n$,
$$d(1^n, v) \leq C n^2.$$

The only letter in $Q = \{a,b,c\}$ which swaps the elements of $A$ is $c$. The other two letters fix $0$ and $1$. Hence, a word $w \in Q^*$ fixes $1$ if and only if it has an even number of $c$'s.

For each $n > 0$, there is a reduced word ending in $a$ or $c$ which contains an odd number of $c$'s. We can take, e.g. $abab...abc$ or $baba...abc$. By Lemma~\ref{lemma:bellaterra_transitive}, if $w$ is any reduced word of length $n$ ending in $a$ or $c$, then its orbit under $\tauw_1$ contains some word which does not fix $1$. Hence, by Lemma~\ref{lemma:orbit_nofix}, $w$ does not fix the infinite word $111\ldots = x^\infty$.

Given a number $n\geq 1$, we have $\abs{A^n} = 2^n$, and there are $2^{n+1} - 1$ reduced words of length $n$ or less which end in $a$ or $c$. By the pigeonhole principle, there must be two such words, $v,w$ with $\act{v}(1^n) = \act{w}(1^n)$. We may assume $\len(v) \leq \len(w)$. Since $a^2$, $b^2$, and $c^2$ all act trivially on $A^*$, reversing a word inverts its action on $A^*$. Let $u$ be the reduced word formed by cancelling pairs of repeated letters in $\ol v w$. Then,
\begin{gather*}
\act{u}(1^n) = \act{\ol v w}(1^n) = \act{\ol v v}(1^n) = 1^n,
\\
\llap{\and}
\len(u) \leq \len(\ol v w) = \len(v) + \len(w) \leq 2n.
\end{gather*}
Since $v \neq w$, $u$ is not the empty word. We assumed that $\len(v) \leq \len(w)$, so the last letter of $w$ is not cancelled. Hence $u$ also ends in in $a$ or $c$, and therefore
$$
\act{u}(111\dots) \neq 111\dots.
$$

Let $k$ be the maximal integer such that $\act{u}{(1^k)} = 1^k$. We know $k \geq n$ and $\act{u}{(1^{k+1})} = 1^k0$. So, letting $s = 1^{k-n}$, $t = 1^{n+1}$, and $t' = 1^n0$, we have
\begin{align*}
st'
= \act{u}(st)
= \act{u}s\,\act{u'}(t)
= s\,\act{u'}(t),
\end{align*}
where $u' = u^s$. So we have
\begin{gather*}
\act{u'}(1^{n+1}) = 1^{n}0,
\\
\llap{\and}
\len(u') = \len(u) \leq 2n.
\end{gather*}

This construction works for all $n \geq 1$. That is, for every $n \geq 1$, there exists a $u_n \in Q^*$ with $\len(u_n) \leq 2n$ and $\act{u_n}{(1^{n}0)} = 1^{n+1}$.

We now prove by induction on $n$ that for every $s \in A^n$, there is a $w \in Q^*$ with $\len(w) \leq n^2$ such that $\act{w}{s} = 1^n$. The base cases $n = 0$ and $n = 1$ are trivial. For the inductive step, consider any $n \geq 1$. Given $s \in A^{n+1}$, let $s'$ be $s$ with the last digit removed. By the induction hypothesis know there is a word $w$ with $\len(w) \leq n^2$ such that $\act{w}{s'} = 1^n$. Then either $\act{w}{s} = 1^{n+1}$ or $\act{w}{s} = 1^n0$. In the first case, we are done. In the second case, $\act{u_n w}{s} = 1^{n+1}$, and $\len(u_n w) \leq 2n + n^2 \leq (n+1)^2$, so we are done.

So, we have shown that in the graph $B_n = \Gamma_{\cb,n}$, we have $d(1^n,s) \leq n^2$ for every $s \in A^n$. It follows that for any $s,t \in B_n$,
$$d(s,t) \leq d(s,1^n) + d(1^n,t) \leq 2n^2,$$
i.e.~$\diam(B_n) \leq 2n^2.$
\end{proof}


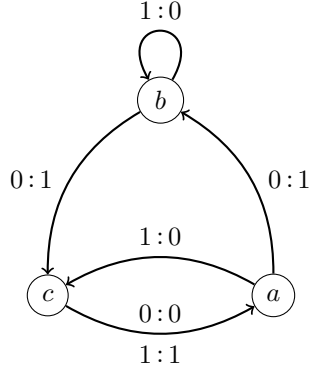
\section{The Aleshin automaton}\label{section:aleshin}

The Aleshin automaton $\ca$ and the Bellaterra automaton $\cb$ are closely related. Indeed, let $\tau_d: \{0,1\}^* \ra \{0,1\}^*$ denote map which swaps every digit of a binary word. Then it is straightforward to check by induction that
\begin{align*}
\tauw_{\ca, a} &= \tauw_d\,\tauw_{\cb, a},\\
\tauw_{\ca, b} &= \tauw_d\,\tauw_{\cb, c}, \\
\llap{\and}
\tauw_{\ca, c} &= \tauw_d\,\tauw_{\cb, c}.
\end{align*}

With this observation, Theorem~\ref{theorem:bellaterra_diameter} has the following corollary.
\begin{cor} \label{cor:aleshin_diameter}
Let $A_n$ denote the $n$-th Aleshin graph. Then $\diam(A_n) = O(n^2)$.
\end{cor}
\begin{proof}
For every pair $q, r \in \{a,b,c\}$, we have
$$
\tauw_{\ca, q}^{-1}\, \tauw_{\ca, r}
= \tauw_{\cb, q}^{-1} \,\tauw_d^{-1} \, \tauw_d \, \tauw_{\cb, r}
= \tauw_{\cb, q} \tauw_{\cb, r}.$$
So, if two words in $\{0,1\}^n$ are separated by a path of length $2$ in the Bellaterra graph $B_n$, they are also separated by a path of length $2$ in the Aleshin graph $A_n$. It follows that two endpoints of an even-length path in $B_n$ are endpoints of a path in $A_n$ of the same length.

For any word $s \in \{0,1\}^n$ there is a path in $B_n$ of length $O(n^2)$ from $1^n$ to $s$. We may assume that this path has even length since $1^n$ has an edge in $B_n$ from itself to itself. This corresponds to a path in $A_n$ of the same length, so for any $s \in\{0,1\}^n$, there is a path in $A_n$ of length $O(n^2)$ from $1^n$ to $s$. Therefore, $\diam(\Gamma_{\ca,n}) = O(n^2),$ as desired.
\end{proof}


\section{Generalizations}\label{section:general}

The proof of Theorem~\ref{theorem:bellaterra_diameter} can be adapted to prove a more general result. In order to generalize to automata with larger alphabets, we need to consider a restricted type of automaton. We say an Mealy automaton $\cm = (Q,A, \tau, \sigma)$ is \emph{cyclic} if it is invertible, and $\gen{\sigma_q \mid q \in Q} = \gen{(x_1\,x_2\,\dots\,x_n)}$, where $\{x_1, x_2, \dots, x_n\} = A$. That is, if its action on $A$ is a cyclic permutation group. In particular, any automaton with $\abs{A} = 2$ is cyclic. This will enable us to reach any word of the form $x^ny$ from $x^{n+1}$ in a short time, as long as we can reach some such word.

We first state and prove the general result with the weakest assumptions under which our argument guarantees polynomial growth of $\diam (\Gamma_{\cm,n})$.

\begin{thm}\label{thm:growth_to_diam} Let $\cm = (Q,A,\sigma,\tau)$ be a cyclic Mealy automaton with $\abs{A}$ prime, and let $\Gamma = \Gamma_{\cm, \infty}$. Suppose there is a letter $x \in A$ and constants $\alpha > 0$, $K > 1$ such that, for sufficiently large~$r$,
\[ \abs{ B_\Gamma(xxx\dots, r) } \geq K^{r^\alpha}.\]
Then there is a constant $C>0$, such that for all $n$,
\[\diam(\Gamma_{\cm,n}) \leq C n^{1 + 1/\alpha}.\]
\end{thm}
\begin{proof}
Let $p = \abs{A}$. By replacing $\cm$ with $\cm \cup \cm^{-1}$ if necessary, we may assume that for every $q \in Q$, there is a $q' \in Q$ with $\sigmaw_{q'} = \sigmaw_q^{-1}$. This replacement adds edges to the Schreier graphs $\Gamma_{\cm, n}$, but only reverses of edges which were already there, so $\diam(\Gamma_{\cm,n})$ and $B_{\Gamma_{\cm,n}}(s,r)$ are unaffected. Then for a word $w \in Q^*$, we define $w^{-1}$ to be $w$, reversed, with each letter $q$ replaced by $q'$, so that $\sigmaw_{w^{-1}} = \sigmaw_w^{-1}$.

Given sufficiently large $n$, pick $r$ such that
\begin{align*}
((\log_K p) n)^{1/\alpha} < r < 2((\log_K p) n)^{1/\alpha}.
\end{align*}
Then $\abs{ B_\Gamma(xxx\dots, r) } > p^n$. By the pigeonhole principle, some two elements of $B_\Gamma(xxx\dots, r)$ have the same first $n$ digits. That is, there are $v,w \in (Q \cup Q^{-1})^*$ with
\begin{gather*}
\len(v), \len(w)
\leq r,
\quad
\act{v}(x^n)
= \act{w}(x^n),
\and
\act{v}(xxx\dots) \neq \act{w}(xxx\dots).
\end{gather*}
So, there is a $u_0  = v^{-1}w \in (Q \cup Q^{-1})^*$ with
\begin{gather*}
\len(u_0) \leq 2r,
\\
\act{u_0}(x^n) = x^n,
\\
\llap{\and}  \act{u_0}(xxx\dots) \neq xxx\dots.
\end{gather*}

There is some smallest value of $k \geq n+1$ such that $\act{u_0}(x^k) \neq x^k$. Let $X_0 = x^{k-n-1}$ and $X = x^n$, so that $x^k = X_0 X x$ and $\act{u_0}{(x^k)} = X_0 X y$ for some $y \in A$ with $y \neq x$. Let $u = u_0^{X_0}$ so in particular, $\len(u) = \len(u_0)$. Then,
\begin{align*}
X_0 X y
=
\act{u_0}{(X_0 X x)}
=
\act{u_0}X_0\,\act{u}{(X x)}
=
X_0\,\act{u}{(X x)},
\end{align*}
so
\begin{align*}
\act{u}{(X x)} = X y.
\end{align*}

Similarly, if $u' = u^X$, we have
\begin{align*}
\act{u}{(X z)} = X\, \act{u'}{z}
\end{align*}
for any $z \in A$. Since $\act{u'}{x} = y \neq x$ and $\cm$ is cyclic, the action of $u'$ on $A$ is a nontrivial cyclic permutation. Since $p = \abs A$ is prime, $u'$ acts transitively on $A$, and therefore $u$ acts transitively on $\setc{Xz}{z \in Aa}$. It follows that for any $z, z' \in A$,
\begin{align*}
d(Xz, Xz') \leq p \,\len(u) \leq 2 p r \leq 4p((\log_K p) n)^{1/\alpha}.
\end{align*}

\noindent
Thus, there is a constant $C$ such that for sufficiently large $n$, we have
$$d(x^nz, x^{n+1}) \leq C n^{1/\alpha}, \quad \text{for all} \ \, b \in A. $$
By increasing the constant if necessary, we can make this true for all~$n$.

 Now let us show by induction on~$n$ that for all $s \in A^n$, we have $d(s,x^n) < C n^{1+1/\alpha}$. The base case $n = 0$ is trivial. For the inductive step, take any $s \in A^{n+1}$, and let $s'$ be its first $n$ letters. We know $d(s',x^n) < C n^{1+1/\alpha}$. There is some word $w\in  (Q \cup Q^{-1})^*$ with $\len(w) = d(s',x^n)$ and $\act{w}{(s')} = x^n$. Then $\act{w}{s} = x^nz$ for some $z \in B$. Thus,
\begin{align*}
d(s,x^{n+1})
&\leq d(s,x^{n}z) + d(x^{n}z,x^{n+1}) \\
&\leq C n^{1+1/\alpha} + Cn^{1/\alpha}\\
&\leq C (n+1)^{1+1/\alpha},
\end{align*}
which completes the induction.

It follows that for any $s, t \in A^n$, $d(s,t) \leq d(s,x^n) + d(x^n, t) \leq 2C n^{1+1/\alpha},$ i.e.,
$$\diam(\Gamma_{\cm,n}) \leq 2C n^{1+1/\alpha}.$$
\end{proof}

In all the cases where we apply this, $\abs{B_\Gamma(xxx\dots, r)}$ will have exponential growth, so we state that case separately.

\begin{cor}\label{cor:expgrowth_to_diam}
Let $\cm = (Q,A,\sigma,\tau)$ be a cyclic Mealy automaton with $\abs{A}$ prime. Let $\Gamma = \Gamma_{\cm,\infty}$. If there is an $x \in A$ and a constant $K > 1$ such that $$\abs{B_\Gamma(xxx\dots, r)} \geq K^r$$ for sufficiently large $r$, then there is a constant $C>0$, such that 
$$\diam(\Gamma_{\cm,n}) \leq C n^2.$$
\end{cor}

It is not always easy to guarantee that $\abs{B_\Gamma(xxx\dots, r)} \geq K^r$ grows quickly, so we prove an additional result based on the size of orbits of $\tauw_x$ in $Q^n$. Loosely, if the orbits grow quickly enough, it must be because there are enough distinct images of words of the form $x^m$.

\begin{thm}\label{thm:orbit_to_diam}
Let $\cm = (Q,A,\sigma,\tau)$ be a reversible\footnote{The assumption that $\cm$ is reversible may be lifted, if we replace $\abs{ \{\tauw_x^k(w) \mid k \in \zz\} }$ with the length of the (eventual) period of $w, \tau_x(w), \tau_x^2(w), \dots$.} cyclic Mealy automaton with $\abs{A}$ prime. Suppose there is a letter $x \in A$ and constants $K>1$, $\alpha>0$ such that for sufficiently large $n$, there is a $w \in Q^n$ with
$$\abs{\{\tauw_x^k(w) \mid k \in \zz\}} \geq K^{n^\alpha}.$$
Then there is a constant $C>0$, such that for all $n$, we have
$$\diam(\Gamma_{\cm,n}) \,\leq\, C\thinspace n^{1+1/\alpha}.$$
\end{thm}
\begin{proof}
Let $P = \{p \in \nn \mid \text{$p$ prime}, p \leq \abs{Q}\}$. It is easy to see by induction on length that for each $w \in Q^*$, the sequence $w, \tau_x(w), \tau_x^2(w), \dots$ is periodic with period $m$, where $m$ is a product of some powers of primes in $P$. Define letters $q_{i,k}$ via
\begin{align*}
\tau_x^k(w) = q_{0,k}q_{1,k} \dots q_{n,k}.
\end{align*}
If $m_i$ is the period of the sequence $q_{i,0}, q_{i,1}, \dots$, then $m = \gcd(m_0, m_1,\dots,  m_n)$. Let $M = \max_i m_i$. Then each prime power in the prime factorization of $m$ is a factor of some $m_i$, so it is at most $M$. The period $m$ is the product of these prime powers, so $m \leq M^{\abs{P}}$. That is, there is some $i$ such that $m_i \geq m^{1/\abs{P}}$.

Fix that $i$ for the rest of the proof, and let $v$ be the first $i$ letters of $w$. Consider the infinite word $s = \act{v}{(xxx\dots)}$, and let $x_k$ be it's $k$-th letter. Let $l$ be the period of the word $s$. Note that $q_{i,k+1} = q_{i,k}^{x_k}$ and therefore
\begin{align*}
q_{i,k+l} = q_{i,k}^{X},
\end{align*}
where
\begin{align*}
X =x_k x_{k+1} \ldots x_{k+l-1}.
\end{align*}
Since the $x_k$ repeat every $l$ letters, we have $q_{i,k+l} = q_{i,k}^X$, and $q_{i,k+2l} = q_{i,k+l}^X$, and so on. Let $F = \abs{Q}!$. Then $X^F$ acts trivially on $Q$, and hence $q_{i,k+Fl} = q_{i,k+Fl}$. This is true for each $k$, so the $q_{i,k}$ have period $m_i \leq Fl$. Thus, the word $s = \act{v}{(xxx\ldots)}$ has period $l \geq m^{1/\abs{P}}/F$.

Now let $n$ be sufficiently large, so that there is a word $w \in Q^n$ whose orbit under $\tauw_x$ has size $m \geq K^{n^\alpha}$. Then, from the above, for some $v \in Q^*$ with $\len(v) \leq n$, the word $s = \act{v}{xxx\dots}$ has period
\begin{align*}
l \geq \frac{1}{F} K^{ n^\alpha / \abs{P}} \geq \wt K ^{n^\alpha},
\end{align*}
where we fix some $1 < \wt K < K^{1/\abs{P}}$, and the last inequality holds for sufficiently large $n$.

Let $v_k = \tau_x^k(v)$. Then $\act{v_k}{xxx\dots}$ is a shift of $s$, and since $s$ has period $l$ there are $l$ distinct such shifts. So, since each $v_k$ satisfies $\len(v_k) \leq n$, the set $\{\act{w}xxx\dots \mid w \in Q^*, \len(w) \leq n \}$ has at least $l \geq \wt K ^{n^\alpha}$ elements. It follows that $\abs{B_\Gamma(xxx\dots, n)} \geq \wt K ^{n^\alpha}$, where $\Gamma = \Gamma_{\cm,\infty}$

So Theorem~\ref{thm:growth_to_diam} applies, and there is a constant $C$ such that $\diam(\Gamma_{\cm,n}) \leq C n^{1 + 1/\alpha}.$
\end{proof}

We also state the following special case, which is a simple way to apply the theorem.

\begin{cor}\label{cor:tree_orbit_to_diam}
Let $\cm = (Q,A,\sigma,\tau)$ be a reversible cyclic Mealy automaton with $\abs{A}$ prime. Suppose there is some $a \in A$, and some $d \geq 2$, such that there is a $d$-regular subtree $\T \subseteq Q^*$ such that $\tauw_a$ acts spherically transitively on $\T$. Then there is a constant $C>0$, such that for all~$n$, we have 
$$\diam(\Gamma_{\cm,n}) \, \leq \, Cn^2.$$
\end{cor}


\section{Cotransitive cyclic automata}\label{section:cotransitive}

The simplest way for the conditions in Corollary~\ref{cor:tree_orbit_to_diam} to be satisfied is when some $\tauw_a$ acts spherically transitively on the entire tree $Q^*$. With that in mind, we make the following definitions.

We say an invertible Mealy automaton $\cm = (Q, A, \sigma, \tau)$ is \emph{$q$-transitive} if the tree automorphism $\sigmaw_q : A^* \ra A^*$ is spherically transitive. We say $\cm$ is \emph{transitive}\footnote{A more natural definition of this term might be that the $\sigmaw_q$ together act transitively on each level of $A^*$, but that is too general for our purposes} if it is $q$-transitive for some $q \in Q$. We say $\cm$ is \emph{cotransitive} if its dual is transitive.

Then, according to Corollary~\ref{cor:tree_orbit_to_diam}, we have
\begin{cor}\label{cor:cotransitive_to_diam}
Let $\cm = (Q,A,\sigma,\tau)$ be a reversible cyclic cotransitive Mealy automaton with $\abs{A}$ prime. Then $\diam(\Gamma_{\cm,n}) = O(n^2)$.
\end{cor}

We do not know a general method for determining whether a tree automorphism given by an automaton is transitive, but there are special cases where checking it is easier. For example, \cite{spherical_transitivity} gives a generalization of Lemma~\ref{lemma:binary_transitive} to all cyclic automata:

\begin{lemma}\label{lemma:cyclic_transitive}
Let $\cm = (Q, A, \tau, \sigma)$ be a cyclic automaton, with $\abs{A} = m$. Then there is a cyclic permutation $\rho$ of $A$, such that for each $q \in Q$ there is a $k_q$ s.t.~$\sigma_q = \rho^k$. Recursively define
\begin{align*}
\chi(q) = k_q + t \sum_{x \in A}{\chi(\tau_x(q))} \in \zz_m[[t]]
\end{align*}
Then $\sigmaw_q$ acts transitively on $A^*$ if and only if each coefficient of $\chi(q)$ is a generator of $\zz_m$.
\end{lemma}

An automaton is called \emph{cocyclic} if its dual is cyclic.  Now observe that the power series~$\chi(q)$ for $q \in Q$ satisfy a recursive linear relation, which can be solved to write each $\chi(q)$ as a rational function.  This implies:

\begin{cor}\label{cor:cyclic_transitive_alg}
Given a (co)cyclic Mealy automaton $\cm = (Q,A,\tau,\sigma)$, there is an algorithm to determine whether it is (co)transitive.
\end{cor}

\begin{figure}
\input{cotransitive32.tex}
\label{figure:cotransitive32}
\end{figure}

For example, it is straightforward to check that, there are 16 cocyclic invertible $(3,2)$-automata, and only four are cotransitive. These four are the automata pictured in Figures~\ref{subfig:32_abelian_start}--\ref{subfig:32_abelian_end}, i.e., automata number 956, 2396, 870, and 2294 in \cite{3_state_automata}.\footnote{Note that \cite{3_state_automata} does not distinguish between an automaton and its inverse. We do, so some of our automata are actually inverses of the automata described there.}


\section{Further examples}\label{section:examples}

\begin{Example} \label{example:nonabelian_23}
It can be verified that, except for the automata pictured in Figure~\ref{figure:cotransitive32}, every invertible $(3,2)$-automaton is not cotransitive: for each automaton, simply find orbits of each $\tauw_a$ which are proper subsets of $A^n$ for some $n$. In this case, it suffices to take $n = 4$.

The final automaton $\cm$ pictured in Figure~\ref{subfig:32_nonabelian}, which is automaton number 2372 in \cite{3_state_automata}, is not cocyclic. So, we do not have a mechanical procedure to prove it is cotransitive. It turns out, however, that there is an automorphism $\kappa : \{a,b,c\}^* \ra \{a,b,c\}^*$ such that $\kappa^{-1} \tauw_{\cm,1} \kappa$ can be computed by a cyclic automaton. Indeed, one can take $\kappa = \tauw_{\cC,x}$, where $\cC$ is the automaton in Figure~\ref{figure:conjugator}. Then we can compute the power series $\chi(\kappa^{-1} \tauw_{\cm,1} \kappa)$, and see directly that its coefficients are nonzero. At that point, Corollary~\ref{lemma:cyclic_transitive} implies that $\kappa^{-1} \tauw_{\cm,1} \kappa$ acts transitively on $Q^*$, and hence so does $\tauw_{\cm,1}$.
\end{Example}

\begin{figure}

\begin{tikzpicture}

	\node at (0cm,0) [state] (a) {$a$};
	\node at (3cm,0) [state] (b) {$c$};
	\node at (6cm,0) [state] (c) {$b$};
	
	\draw
		(a) edge[transition, bend right, /from to = {$x$}{$x$}{above}] (b)
		(b) edge[transition, bend right, /from to = {$x$}{$x$}{above}] (a)

		(c) edge[transition, /wide loop at = 90, /from to = {$y$}{$x$}{near start, above right}, /from to = {$x$}{$y$}{near end, above left}] (c)
		(a) edge[transition, /wide loop at = 180, /from to = {$y$}{$y$}{left}] (a)
		(b) edge[transition, /wide loop at = 0, /from to = {$y$}{$y$}{right}] (b)
		
	;
\end{tikzpicture}

\caption{An automaton to conjugate $\cm_{2372}$ into a cocyclic automaton.}
\label{figure:conjugator}
\end{figure}
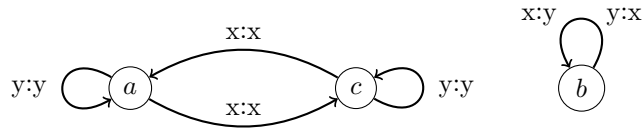

So, we have sketched a proof of the following:
\begin{prop} \label{prop:32_list}
The cotransitive invertible $(3,2)$-automata are precisely the five automata pictured in Figure~\ref{figure:cotransitive32}, up to relabeling of $A$ and $Q$.
\end{prop}

\begin{Example} \label{example:23_automaton}
Of course, there are automata which are not cotransitive, but still satisfy the conditions of Corollary~\ref{cor:tree_orbit_to_diam}. As we saw, one example is the Bellaterra automaton. A natural and easy to analyze example is the automaton $\cm = (Q, A, \tau, \sigma)$ that implements division by $3$ modulo~$2^n$. (This is automaton number 924 in \cite{3_state_automata}. See \cite{solvable_automata} for more on this construction and related ones.) We will also see that its Schreier graphs do not form a family of expanders.

A quick way to define this automaton is that for $a,b \in Q = \{0,1,2\}$ and $x,y \in A = \{0,1\}$, we have $a \edgetol*{x \ft y} b$ if and only if $$a + 3y = x + 2b.$$ This automaton is pictured in Figure~\ref{figure:23_automaton}. Note that for convenience we abuse notation slightly and call two of the states, $0$ and $1$, by the same name as the letters in the alphabet.

\begin{figure}
\begin{tikzpicture}

	\node at (0cm,0) [state] (0) {$0$};
	\node at (2cm,0) [state] (1) {$1$};
	\node at (4cm,0) [state] (2) {$2$};
	
	\draw
		(0) edge[transition, /from to = {0}{0}{left}, /wide loop at = 180, min distance = 1cm] (0)
		(0) edge[transition, bend right, /from to = {1}{1}{below}] (1)

		(1) edge[transition, bend right, /from to = {1}{0}{above}] (0)
		(1) edge[transition, bend right, /from to = {0}{1}{below}] (2)

		(2) edge[transition, /from to = {1}{1}{right}, /wide loop at = 0, min distance = 1cm] (2)
		(2) edge[transition, bend right, /from to = {0}{0}{above}] (1)
	;
\end{tikzpicture}
\caption{}
\label{figure:23_automaton}
\end{figure}

By assumption, if $a \edgetol*{x \ft y} b$, then for any $x' \in \zz/2^{n-1}\zz$, we have the following equalities in $\zz/2^{n}\zz$:
\begin{align*}
x + 2b &= a + 3y\\
x + 2x' - a &= 3y + 2x' - 2b\\
\frac{(x + 2x') - a}{3} &= y + 2\frac{x' - b}{3}.\\
\end{align*}
That is, if $x$ is the least significant binary digit of a number $X \in \zz/ 2^n \zz$, and $x' \in \zz/ 2^{n-1} \zz$ is the number corresponding to the rest of its digits, then the least significant digit of $(X-a)/3$ is $y$, and the rest of the digits are given by $(x'-b)/3$. It follows that if we identify a number $x \in \zz/2^n \zz$ with its binary representation in $\{0,1\}^n$ (with the least significant digit on the left), then we have, for each $a \in \{0, 1, 2\}$,
\begin{align*}
\sigmaw_a(x) = \frac{x-a}{3}.
\end{align*}

By a symmetric argument, the dual of this automaton implements division by $2$ modulo $3$. Phrasing this in terms of the original automaton $\cm$, we interpret a length-$m$ word in $\{0, 1, 2\}$ as the representation of a number modulo $3^m$ written in ternary with the least significant digit on the right. Then for each $x \in \{0,1\}$,
\begin{align*}
\tauw_x(a) &= \frac{a-x}{2}.
\end{align*}

In particular, $\tauw_0$ divides a number by 2. Since $2$ generates the multiplicative group $(\zz/3^m\zz)^*$, that group is an orbit of $\tauw_0$. So for every $m$, there is an orbit of $\tauw_0$ in $Q^m$ of size $2\cdot 3^{m-1}$. By Theorem~\ref{thm:orbit_to_diam}, it follows that $\diam(\Gamma_{\cm,n}) = O(n^2).$ In fact, it can be checked explicitly that $\diam(\Gamma_{\cm,n}) = O(n)$. This can be seen from the observation that the sequence of applications of $\sigmaw_1$, $\sigmaw_2$, and $\sigmaw_3$ necessary to send the binary number $x$ to $00\ldots0$ is essentially the representation of $x$ in base $3$.

However, the group $G_\cm = \gen{\sigmaw_{0}, \sigmaw_{1}, \sigmaw_{2}}$ is generated by $\mu = \sigmaw_{0}^{-1}$ and $\alpha = \sigmaw_{1}^{-1} \sigmaw_{0}$, which are multiplication by $3$ and addition of $1$, respectively. E.g., $\sigmaw_{2} = \mu^{-1} \alpha^{-2}$. It follows that the group action factors through the group of upper-triangular $2$ by $2$ matrices via
\begin{align*}
\mu \mapsto
\begin{pmatrix}
3 & 0 \\
0 & 1
\end{pmatrix}
\qquad
\alpha \mapsto
\begin{pmatrix}
1 & 1 \\
0 & 1
\end{pmatrix}
\end{align*}
This group is solvable, and therefore amenable. It follows that its Schreier graphs with respect to a fixed set of generators cannot be expanders \cite[3.3.7]{lubotzky_book}. So, the family $\{\Gamma_{\cm,n}\}_{n=1}^\infty$ is not a family of expanders.
\end{Example}

So, there are automata to which our general results apply, but whose Schreier graphs do not form a family of expanders. More work is necessary to find sufficient conditions for when an automaton gives rise to a family of expanders.

\begin{Example}
It can be checked by a computation that there are no cotransitive invertible $(4,2)$-automata. It turns out it is enough to check the actions of the $\tau_a$ on $Q^4$.
\end{Example}

\begin{Example} \label{example:52}
There are seven $(5,2)$-automata which are not cocyclic, but act transitively on~$Q^{10}$. Of these, just one is bireversible, as the Aleshin and Bellaterra automata are. It is pictured in Figure~\ref{figure:25_bireversible}. Unlike the automaton in Figure~\ref{subfig:32_nonabelian}, it is unlikely that there is an automatic automorphism~$\kappa: Q^* \ra Q^*$ such that $\kappa^{-1} \tau_{0} \kappa$ is implemented by a cocyclic automaton. We have checked that if there is such a $\kappa$, the automaton implementing it would need to have at least $48668$ states.
\end{Example}

\begin{figure}
\begin{tikzpicture}[scale = 2]

	\node at (-54:1cm) [state] (a) {$a$};
	\node at (18:1cm) [state] (b) {$b$};
	\node at (90:1cm) [state] (c) {$c$};
	\node at (162:1cm) [state] (d) {$d$};
	\node at (-126:1cm) [state] (e) {$e$};
	
	\draw
		(a) edge[transition, /from to = {0}{1}{right}] (b)
		(b) edge[transition, /from to = {0}{0}{above right}] (c)
		(c) edge[transition, /from to = {0}{0}{above left}] (d)
		(d) edge[transition, /from to = {0}{0}{left}] (e)
		(e) edge[transition, /from to = {0}{1}{below}] (a)

		(b) edge[transition, bend right, /from to = {1}{1}{}] (e)
		(e) edge[transition, bend right, /from to = {1}{0}{}] (b)
		(c) edge[transition, bend left, <->, /from to = {1}{1}{}] (d)
		(a) edge[transition, /from to = {1}{0}{right, near end}, /wide loop at = -54, min distance = .5cm] (a)
	;
\end{tikzpicture}
\caption{The only candidate to be a cotransitive bireversible $(5,2)$-automaton.}
\label{figure:25_bireversible}
\end{figure}
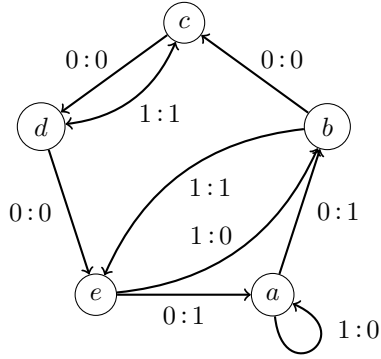


\section{Remarks and further work}\label{section:remarks}

\subsectionn{}\label{remark:non_invertible}
 The results in Section~\ref{section:general} can be extended to non-invertible Mealy automata as well.  Since we are primarily interested only in regular graphs, we prove only the simpler case. For example, to state Theorem~\ref{thm:growth_to_diam} more generally, one needs to consider the size of balls in $\Gamma_{\cm,\infty}$ defined in terms of directed paths, but the result about the diameter of $\Gamma_{\cm, n}$ still needs diameter to be defined in terms of undirected paths.

\subsectionn{}\label{remark:bellaterra_implies_aleshin}
As noted in the proof of Corollary~\ref{cor:aleshin_diameter}, any product of two generators of the Bellaterra group $G_\cb = \gen{\sigmaw_{\cb,a},\sigmaw_{\cb,b},\sigmaw_{\cb,c}}$ belongs to the Aleshin group $G_\ca = \gen{\sigmaw_{\ca,a},\sigmaw_{\ca,b},\sigmaw_{\ca,c}}$. We used this fact to show that, since the Bellaterra graphs have small diameter, so do the Aleshin graphs. In fact, it can also be used to show that if the Bellaterra graphs form a (two-sided) expander family, so do the Aleshin graphs. In other words, Conjecture~\ref{conj:bellaterra_expanders} implies Conjecture~\ref{conj:aleshin_expanders}.

\subsectionn{}
A tree automorphism $\alpha:A^* \ra A^*$ is spherically transitive if and only if it is conjugate in $\Aut(A^*)$ to the adding machine $\rho$, i.e., the automorphism which interprets a word in $A^n$ as the base-$\abs{A}$ representation of a number modulo $\abs{A}^n$, and adds one to that number. The adding machine is an automatic automorphism, e.g., the binary adding machine is pictured in Figure~\ref{figure:adding_automaton}.

One might hope that whenever an automatic automorphism $\alpha \in \aAut(A^*)$ is conjugate to $\rho$ in $\Aut(A^*)$, it is also conjugate to $\rho$ in $\aAut(A^*)$. If so, we would have an algorithm for determining whether a given automatic automorphism is transitive. In fact, since we can enumerate the transitive cyclic automata, it would be enough if every transitive $\alpha \in \aAut(A^*)$ were conjugate in $\aAut(A^*)$ to some cyclic automorphism.

However, Example~\ref{example:52} suggests that, in the dual of the automaton in Figure~\ref{figure:25_bireversible}, $\sigmaw_{0}$ is transitive but not conjugate in $\aAut(A^*)$ to any cyclic automaton, in particular to $\rho$. However, we  prove that $\sigmaw_0$ is not conjugate to a cyclic automaton, nor prove that it is actually transitive.

\begin{problem} \label{problem:cotransitive_conj_to_cyclic}
Exhibit a transitive $\alpha \in \aAut(A^*)$ which is not conjugate (in $\aAut(A^*)$) to a cyclic $\beta \in \aAut(A^*)$. (Or prove that there is no such $\alpha$.)
\end{problem}

\begin{problem} \label{problem:conj_to_cyclic_characterization}
Characterize the automorphisms in $\aAut(A^*)$ which are conjugate in $\aAut(A^*)$ to a cyclic automorphism.
\end{problem}

We can, however, exhibit a cyclic $\alpha \in \aAut(A^*)$ which is not conjugate in $\aAut(A^*)$ to the adding machine $\rho$:

\begin{prop}\label{prop:not_conj_to_adder}
Let $\cm = (Q,A,\tau,\sigma)$ be the dual of the automaton in Figure~\ref{subfig:32_abelian_start}, where $Q = \{0,1\}$ and $A = \{a,b,c\}$. Then $\sigmaw_1$ acts transitively on $A^*$, but there is no $\kappa \in \aAut(A^*)$ such that $\kappa^{-1} \sigmaw_1 \kappa = \rho$
\end{prop}
\begin{proof}[Sketch of proof:]
Given an eventually periodic word $w \in A^*$, we let $h(w)$ denote the smallest number $n$ such that $w$ is periodic after the first $n$ letters.

Note that if $\rho$ is the adding machine, then for any eventually periodic word $v \in A^\infty$, we have
$$h(\rho^n(v)) = O(\log n).$$
Moreover, after a finite number of steps, the periodic part of $\rho^n(v)$ stabilizes. It follows that for any $\kappa \in \aAut(A^*)$, we have
$$h(\kappa \rho^n(v)) = O(\log n)$$
and since this applies to any $v$,
$$h(\kappa \rho^n \kappa^{-1}(v)) = O(\log n)$$

On the other hand, taking $\alpha = \sigmaw_1$, we can check that if we read $w \in A^*$ as a balanced ternary number (with $a = -1, c = 0, b = -1$), we have
\begin{align*}
\alpha(w) = \frac{w - 1}{2}.
\end{align*}
It follows that for $w = ccc\ldots$,
\begin{align*}
h(\alpha^{-n}(w)) \sim (\log_3 2)n
\end{align*}

Thus $\alpha^{-1}$ and $\rho$ are not conjugate in $\aAut(A^*)$. It is easy to check that $\rho$ and $\rho^{-1}$ are conjugate, so $\alpha$ and $\rho$ are not conjugate in $\aAut(A^*)$.
\end{proof}

\subsectionn{}

More generally, an open problem is the classification of conjugacy classes in $\aAut(A^*)$. The conjugacy classes of $\Aut(A^*)$ can be described in terms of orbit trees \cite{conjugation_in_tree_auts}. This tree captures the information about the orbits of an automorphism $\alpha \in \Aut(A^*)$, e.g. a ray with few branches in the orbit tree corresponds to a sequence of quickly growing orbits. Information about this tree can tell us whether we can apply, e.g., Theorem~\ref{thm:orbit_to_diam}.

Of course, not all orbit trees arise from elements of $\aAut(A^*)$, since there are uncountably many. Moreover, not all automatic automorphisms with the same orbit tree are conjugate in $\aAut(A^*)$, as seen in Proposition~\ref{prop:not_conj_to_adder}.

In \cite{conjugacy_problem}, the problem is solved for \emph{bounded automorphisms}, and more generally automorphisms with \emph{finite orbit-signalizer}. Such ``small'' automorphisms are unlikely to give expanders, so we are interested in the other end of the spectrum, automorphisms with many nontrivial sections on every level.

\subsectionn{}

\emph{Automaton groups}, i.e., groups of the form $G_\cm$ for some Mealy automaton $\cm$, are of independent interest in group theory.\footnote{Note that the term \emph{automatic group} has a different meaning in the literature, one we do not use in this paper.} A famous example is the \emph{Grigorchuk group}, which is the first known group whose growth function is intermediate between polynomial and exponential (see \cite{growthintro, solved_unsolved}). For more on automaton groups, see \cite{branch_groups,automata_dynamical,Nek}.

\subsectionn{}\label{remark:amenable_nonexpanders}

The structure of $G_\cm$ as an abstract group can give us information on whether or not the graphs $\Gamma_{\cm,n}$ form a family of expanders. For example, if $G_\cm$ is amenable then $\{\Gamma_{\cm,n}\}_{n=1}^\infty$ is not a family of expanders \cite[3.3.7]{lubotzky_book}. We already used this fact in Example~\ref{example:23_automaton} to show that the Schreier graphs of the automaton defined there are not expanders.

\subsectionn{}\label{remark:property_t_expanders}

On the other hand, sometimes the structure of $G_\cm$ is enough to guarantee that $\{\Gamma_{\cm,n}\}_{n=1}^{\infty}$ is a family of expanders.
Notably, if $\Gamma_1, \Gamma_2, \dots$ are Schreier graphs (with respect to a fixed generating set) of a group with \emph{Kazhdan property $(T)$}, and $\abs{\Gamma_i} \ra \infty$, then these graphs must form a family of expanders \cite[3.3.4]{lubotzky_book}. This fact was used by Margulis to give the first explicit construction of expanders \cite{margulis}.

In \cite{square_complexes}, it was shown that there are Mealy automata $\cm$ for which $G_\cm$ has property $(T)$, so Mealy automata can be used to construct expander families.  The groups $G_\ca$ and $G_\cb$ do not have property $(T)$, so this approach is not sufficient to prove Conjectures~\ref{conj:aleshin_expanders} and~\ref{conj:bellaterra_expanders}.

\subsectionn{}\label{ramanujan_lifts}

In a recent preprint, \cite{all_degree_expanders}, the ideas of \cite{random_lifts_expansion} were extended to construct families of bipartite Ramanujan graphs (i.e., expander graphs with optimal spectral gap) of arbitrary degree. The construction uses a new technique to pick a particular 2-lift of a graph which does not introduce any new large eigenvalues. We should note that this construction is not \emph{very explicit}, in the sense given above.

\subsectionn{}\label{asymptotic}

In~\cite[Section~10]{some}, Grigorchuk shows that in a certain formal sense, the Aleshin and Bellaterra automata are examples of \emph{asymptotic expanders}, thus giving further evidence to Conjectures~1.2 and~1.3.  He also states these conjectures as open problems, and suggests that a sequence of Schreier graphs constructed by a finite automaton cannot be Ramanujan.

\bigskip

\noindent
{\bf Acknowledgements} \ We are grateful to Slava Grigorchuk, Volodia Nekrashevych, Yehuda Shalom and Terry Tao for helpful discussions.  The second author was partially supported by the~NSF.


\end{document}

%% file: preamble.tex
\usepackage{amsmath,amsfonts,amssymb,amsthm}

\usepackage{mathtools}

\usepackage[explicit]{titlesec}

\makeatletter
\def\@secnumpunct{. }
\makeatother

\titleformat{\subsection}[block]
{}{\thetitle. }{0pt}{\bf{#1.}}

\newcounter{subsectionn}
\titleclass{\subsectionn}{straight}[\section]
\titleformat{\subsectionn}{\textbf}{\thesubsectionn}{1em}{}
\renewcommand{\thesubsectionn}{\thesection.\arabic{subsectionn}}

\titleformat{\subsectionn}[runin]{}{\thetitle. }{0pt}{}
\titlespacing{\subsectionn}{0pt}{*1.5}{*1.5}

\newtheorem{thm}{Theorem}[section]
\newtheorem{lemma}[thm]{Lemma}

\newtheorem{cor}[thm]{Corollary}
\newtheorem{prop}[thm]{Proposition}
\newtheorem{conj}[thm]{Conjecture}

\newtheorem{problem}[thm]{Problem}

\theoremstyle{definition}
\newtheorem{Example}[thm]{Example}

\newtheorem{Defn}[thm]{Definition}

\def\zz{\mathbb Z}

\def\nn{\mathbb N}

\def\ca{\mathcal A}
\def\cb{\mathcal B}
\def\cC{\mathcal C}

\def\cm{\mathcal M}

\def\ve{\varepsilon}

\newcommand{\e}{\ve}

\def\wh{\widehat}
\def\wt{\widetilde}
\def\ol{\overline}

\def\<{\langle}
\def\>{\rangle}

\newcommand{\ra}{\rightarrow}
\newcommand{\ua}{\uparrow}

\renewcommand{\and}{\quad\text{and}\quad}
\def\id{{\rm id}}
\newcommand{\len}{\ell}

\newcommand{\edgeto}{\mathrel{\longrightarrow}}

\newcommand{\setc}[3][]{\left\{ #2 \,\middle|\, #3 \vphantom{#1|}\right\}}

\newcommand{\gen}[1]{\left\< #1 \right\>}
\newcommand{\abs}[2][]{\left|#2 \vphantom{#1|}\right|}

\DeclareMathOperator{\Aut}{Aut}

\DeclareMathOperator{\diam}{diam}
\DeclareMathOperator{\da}{\downarrow}

%% file: bellaterra_graphs.tex
\captionsetup[subfloat]{labelformat=empty}
\pgfmathsetmacro{\lR}{.75}
\tikzmath{
	int bendangle;
	\bendangle0 = 90;
	\bendangle1 = 40;
	\bendangle2 = 20;
	\bendangle3 = 15;
	\bendangle4 = 10;
}
\subfloat[$B_0$]{
\raisebox{1cm}{
\begin{tikzpicture}[graph edge/.add style={}{-, font = \scriptsize}, scale = 1.5]
	\node at (0,0) [vertex] (x) {};
	
    \draw 
        (x) edge[loop above, /eq = left, a] (x)
           	edge[/loop at = 90 + 120,/eq=left, b]  (x)
           	edge[/loop at = 90 - 120,/eq= right, c] (x);
\end{tikzpicture}
}
}
\subfloat[$B_1$]{
\raisebox{1cm}{
\begin{tikzpicture}[graph edge/.add style={}{-, font = \scriptsize},scale = 1.5]
	\node at (\lR,0) [vertex] (0) {};
	\node at (-\lR,0) [vertex] (1) {};
	
    \draw
        (0) edge[/dir bend = {180}{-\bendangle1}, /eq = above, c] (1)
        (0)	edge[/loop at = 90 + 120,/eq=left, a] (0)
        (1)	edge[/loop at = 90 + 120,/eq=left, b] (1)
        (0)	edge[/loop at = 90 - 120,/eq=right, b] (0)
        (1)	edge[/loop at = 90 - 120,/eq=right, a] (1)
        ;
\end{tikzpicture}
}
}
\subfloat[$B_2$]{
\begin{tikzpicture}[graph edge/.add style={}{-, font = \scriptsize},scale = 1.5]
	\node at (\lR,\lR) [vertex] (00) {};
	\node at (-\lR,\lR) [vertex] (01) {};
	\node at (\lR,-\lR) [vertex] (10) {};
	\node at (-\lR,-\lR) [vertex] (11) {};

    \draw 
        (00)	edge[/dir bend = {-90}{-\bendangle1},/eq=right, c] (10)
        (01)	edge[/dir bend = {-90}{\bendangle1}, in = 90 - \bendangle1,/eq=left, c] (11)
        (01)	edge[/loop at = 90 + 120,/eq=left, a] (01)
        (11)	edge[/loop at = 90 + 120,/eq=left, b] (11)
        (00)	edge[/loop at = 90 - 120,/eq=right, a] (00)
        (10)	edge[/loop at = 90 - 120,/eq=right, b] (10)
        (00) 	edge[/dir bend = {180}{-\bendangle2}, /eq = above, b] (01)
        (10)	edge[/dir bend = {180}{-\bendangle2}, /eq = above, a] (11)
        ;
\end{tikzpicture}
}

\subfloat[$B_3$]{
\begin{tikzpicture}[graph edge/.add style={}{-, font = \scriptsize},scale = 1.5]
	
	\begin{scope}[xscale = 1.5, yshift = .5*\lR cm]
		\node at (\lR,\lR) [vertex] (000) {};
		\node at (-\lR,\lR) [vertex] (001) {};
		\node at (\lR,-\lR) [vertex] (010) {};
		\node at (-\lR,-\lR) [vertex] (011) {};
	\end{scope}

	\node at (\lR,\lR) [vertex] (100) {};
	\node at (-\lR,\lR) [vertex] (101) {};
	\node at (\lR,-\lR) [vertex] (110) {};
	\node at (-\lR,-\lR) [vertex] (111) {};
	
    \draw 
        (001)	edge[/dir bend = {-60}{-\bendangle0}, /eq={left, near start}, c, min distance = .5cm] (101)
        (000)	edge[/dir bend = {-120}{\bendangle0}, /eq={right, near start}, c, min distance = .5cm] (100)
        (010) 	edge[/dir bend = {180}{-\bendangle3}, /eq = {below, near end}, c] (111)
        (110) 	edge[/dir bend = {180}{-\bendangle3}, /eq = {below, near start}, c] (011)
        (011)	edge[/loop at = 90 + 120 - 10,/eq=left, a] (011)
        (111)	edge[/loop at = 90 + 120 - 10,/eq=left, b] (111)
        (010)	edge[/loop at = 90 - 120 + 10,/eq=right, a] (010)
        (110)	edge[/loop at = 90 - 120 + 10,/eq=right, b] (110)
        (000) 	edge[/dir bend = {180}{-\bendangle3}, /eq = above, a] (001)
        (100) 	edge[/dir bend = {180}{-\bendangle3}, /eq = below, b] (101)
        (000)	edge[/dir bend = {-90}{-\bendangle2},/eq={right, pos = .6}, b] (010)
        (100)	edge[/dir bend = {-90}{-\bendangle2},/eq={left, pos = .4}, a] (110)
        (001)	edge[/dir bend = {-90}{\bendangle2},/eq={left, pos = .6}, b] (011)
        (101)	edge[/dir bend = {-90}{\bendangle2},/eq={right, pos = .4}, a] (111)
        ;
\end{tikzpicture}
}
\subfloat[$B_4$]{
\begin{tikzpicture}[graph edge/.add style={}{-, font = \scriptsize}, scale = 1.5]

	\begin{scope}[xscale = 1.7, yshift = .7*\lR cm]
		\node at (\lR,\lR) [vertex] (0000) {};
		\node at (-\lR,\lR) [vertex] (0001) {};
		\node at (\lR,-\lR) [vertex] (0010) {};
		\node at (-\lR,-\lR) [vertex] (0011) {};
	\end{scope}

	\node at (\lR,\lR) [vertex] (0100) {};
	\node at (-\lR,\lR) [vertex] (0101) {};
	\node at (\lR,-\lR) [vertex] (0110) {};
	\node at (-\lR,-\lR) [vertex] (0111) {};
	
	\begin{scope}[xscale = 1.2, yshift = -.25*\lR cm]
		\begin{scope}[xscale = 1.6, yshift = .7*\lR cm]
			\node at (\lR,\lR) [vertex] (1000) {};
			\node at (-\lR,\lR) [vertex] (1001) {};
			\node at (\lR,-\lR) [vertex] (1010) {};
			\node at (-\lR,-\lR) [vertex] (1011) {};
		\end{scope}

		\node at (\lR,\lR) [vertex] (1100) {};
		\node at (-\lR,\lR) [vertex] (1101) {};
		\node at (\lR,-\lR) [vertex] (1110) {};
		\node at (-\lR,-\lR) [vertex] (1111) {};
	\end{scope}
	
    \draw 
        (0011)	edge[/dir bend={100}{\bendangle0}, /eq=left, c, min distance = .5cm] (1011)
        (0010)	edge[/dir bend={80}{-\bendangle0}, /eq=right, c, min distance = .5cm] (1010)

        (0000) 	edge[/dir bend = {180}{-\bendangle4}, /eq = {above, near start}, c] (1001)
        (0100)	edge[/dir bend = {-95}{-\bendangle4}, /eq={left, near start}, c] (1110)
        (0101)	edge[/dir bend = {-85}{\bendangle4},/eq={right, near start}, c] (1111)
        
        (1000) 	edge[/dir bend = {180}{-\bendangle4}, /eq = {above, near end}, c] (0001)
        (1100)	edge[/dir bend = {-90}{-\bendangle3},/eq={right, near start}, c] (0110)
        (1101)	edge[/dir bend = {-90}{\bendangle3},/eq={left, near start}, c] (0111)

        (0111)	edge[/loop at = 90 + 120 - \bendangle2,/eq=left, a] (0111)
        (0110)	edge[/loop at = 90 - 120 + \bendangle2,/eq=right, a] (0110)
        (1111)	edge[/loop at = 90 + 120 - \bendangle2,/eq=left, b] (1111)
        (1110)	edge[/loop at = 90 - 120 + \bendangle2,/eq=right, b] (1110)

        (0100) 	edge[/dir bend = {180}{-\bendangle4}, /eq = above, a] (0101)
        (0000)	edge[/dir bend = {-90}{-\bendangle3},/eq={left, near end}, a] (0010)
        (0001)	edge[/dir bend = {-90}{\bendangle3},/eq={right, near end}, a] (0011) 
        (1100) 	edge[/dir bend = {180}{-\bendangle4}, /eq = below, b] (1101)
        (1000)	edge[/dir bend = {-90}{-\bendangle3},/eq={right}, b] (1010)
        (1001)	edge[/dir bend = {-90}{\bendangle3},/eq={left}, b] (1011)
        
        (0001)	edge[/dir bend = {-45}{-\bendangle2},/eq={above right}, b, min distance = 0cm] (0101)
        (0000)	edge[/dir bend = {-135}{\bendangle2},/eq={above left}, b, min distance = 0cm] (0100)
        (0010) 	edge[/dir bend = {180}{-\bendangle4}, /eq = {above}, b] (0111)
        (0110) 	edge[/dir bend = {180}{-\bendangle4}, /eq = {above}, b] (0011) 
        
        (1001)	edge[/dir bend = {-45}{-\bendangle2},/eq={above right, near end, inner sep = .8pt}, a, min distance = 0cm] (1101)
        (1000)	edge[/dir bend = {-135}{\bendangle2},/eq={above left, near end, inner sep = .8pt}, a, min distance = 0cm] (1100)
        (1010) 	edge[/dir bend = {180}{-\bendangle4}, /eq = {below, near end}, a] (1111)
        (1110) 	edge[/dir bend = {180}{-\bendangle4}, /eq = {below, near start}, a] (1011)
        ;
\end{tikzpicture}
}
\caption{The Bellaterra graphs}

%% file: bellaterra_automaton.tex
\captionsetup[subfloat]{labelformat=empty}
\subfloat{
\begin{tikzpicture}

	\node at (0:0cm) [state] (c) {$c$};
	\node at (0:3cm) [state] (a) {$a$};
	\node at (60:3cm) [state] (b) {$b$};
	
	\draw
		(a) edge[transition, bend right, /from to = {0}{0}{right}] (b) 
		(b) edge[transition, bend right, /from to = {0}{0}{left}] (c) 
		(c) edge[transition, bend right, /from to = {0}{1}{above}, /from to = {1}{0}{below}] (a) 
		
		(a) edge[transition, bend right, /from to = {1}{1}{above}] (c) 
		(b) edge[transition, /from to = {1}{1}{above}, /wide loop at = 90, min distance = 1cm] (b)
	;
\end{tikzpicture}
}
\hspace{10em}
\subfloat{
\begin{tikzpicture}[]

	\node at (0,2) [state] (0) {$0$};
	\node at (0,0) [state] (1) {$1$};
	
	\draw
		(0) edge[transition, /wide loop at = 90, /from to = {a}{b}{near end, above left},/from to = {b}{c}{near start, above right}] (0) 
		(0) edge[transition, bend right, /from to = {c}{a}{left},] (1) 
		(1) edge[transition, bend right, /from to = {c}{a}{right},] (0)
		(1) edge[transition, /wide loop at = -90, /from to = {b}{b}{near start, below left}, /from to = {a}{c}{near end, below right}] (1)
	;
\end{tikzpicture}
}
\caption{The Bellaterra automaton $\cb$ and its dual $\ol \cb$}

%% file: aleshin_automaton.tex
\begin{tikzpicture}

	\node at (0:0cm) [state] (c) {$c$};
	\node at (0:3cm) [state] (a) {$a$};
	\node at (60:3cm) [state] (b) {$b$};
	
	\draw
		(a) edge[transition, bend right, /from to = {0}{1}{right}] (b) 
		(b) edge[transition, bend right, /from to = {0}{1}{left}] (c) 
		(c) edge[transition, bend right, /from to = {0}{0}{above}, /from to = {1}{1}{below}] (a) 
		
		(a) edge[transition, bend right, /from to = {1}{0}{above}] (c) 
		(b) edge[transition, /from to = {1}{0}{above}, /wide loop at = 90, min distance = 1cm] (b)
	;
\end{tikzpicture}

\caption{The Aleshin automaton $\ca$}

%% file: cotransitive32.tex
\captionsetup[subfloat]{}
\centering

\makebox[0pt]{
\subfloat[2372]{
\begin{tikzpicture}[scale = .75]

	\node at (0:0cm) [state] (c) {$c$};
	\node at (0:3cm) [state] (a) {$a$};
	\node at (60:3cm) [state] (b) {$b$};
	
	\draw
		(a) edge[transition, /from to = {1}{0}{pos = .35, below left}, /from to = {0}{1}{pos = .65, below left}] (b) 
		(b) edge[transition, /from to = {1}{0}{above left}] (c) 
		(c) edge[transition, /from to = {1}{1}{below}] (a) 
		
		(b) edge[transition, bend left, /from to = {0}{1}{above right}] (a) 
		(c) edge[transition, /wide loop at = -150, /from to = {0}{0}{left}] (c) 
	;
\end{tikzpicture}
\label{subfig:32_nonabelian}
}
\subfloat[956]{
\begin{tikzpicture}[scale = .75]

	\node at (0:0cm) [state] (c) {$c$};
	\node at (0:3cm) [state] (a) {$a$};
	\node at (60:3cm) [state] (b) {$b$};
	
	\draw
		(a) edge[transition, /from to = {1}{0}{above right}] (b) 
		(b) edge[transition, /from to = {1}{0}{above left}] (c) 
		(c) edge[transition, /from to = {1}{1}{below}] (a) 
		
		(a) edge[transition, /wide loop at = -30, /from to = {0}{1}{below right}] (a) 
		(b) edge[transition, /wide loop at = 90, /from to = {0}{1}{above}] (b) 
		(c) edge[transition, /wide loop at = -150, /from to = {0}{0}{left}] (c) 
	;
\end{tikzpicture}
\label{subfig:32_abelian_start}
}
\subfloat[2396]{
\begin{tikzpicture}[scale = .75]

	\node at (0:0cm) [state] (c) {$c$};
	\node at (0:3cm) [state] (a) {$a$};
	\node at (60:3cm) [state] (b) {$b$};
	
	\draw
		(a) edge[transition, /from to = {1}{0}{above right}] (b) 
		(b) edge[transition, /from to = {1}{0}{above left}] (c) 
		(c) edge[transition, /from to = {1}{1}{below}] (a) 
		
		(a) edge[transition, /wide loop at = -30, /from to = {0}{1}{below right}] (a) 
		(b) edge[transition, /wide loop at = 90, /from to = {0}{1}{above}] (b) 
		(c) edge[transition, /wide loop at = -150, /from to = {0}{0}{left}] (c) 
	;
\end{tikzpicture}
}
}

\smallskip

\subfloat[870]{
\begin{tikzpicture}

	\node at (0:0cm) [state] (c) {$c$};
	\node at (0:3cm) [state] (a) {$a$};
	\node at (60:3cm) [state] (b) {$b$};
	
	\draw
		(a) edge[transition, pos = .47, /from to = {1}{0}{below left}] (b) 
		(b) edge[transition, pos = .47, /from to = {1}{1}{below right}] (c) 
		(c) edge[transition, /from to = {1}{1}{above}] (a) 
		
		(a) edge[transition, bend left, /from to = {0}{1}{below}] (c) 
		(b) edge[transition, bend left, /from to = {0}{0}{right}] (a) 
		(c) edge[transition, bend left, /from to = {0}{0}{left}] (b) 
	;
\end{tikzpicture}
}
\subfloat[2294]{
\begin{tikzpicture}

	\node at (0:0cm) [state] (c) {$c$};
	\node at (0:3cm) [state] (a) {$a$};
	\node at (60:3cm) [state] (b) {$b$};
	
	\draw
		(a) edge[transition, pos = .47, /from to = {1}{0}{below left}] (b) 
		(b) edge[transition, pos = .47, /from to = {1}{0}{below right}] (c) 
		(c) edge[transition, /from to = {1}{1}{above}] (a) 
		
		(a) edge[transition, bend left, /from to = {0}{1}{below}] (c) 
		(b) edge[transition, bend left, /from to = {0}{1}{right}] (a) 
		(c) edge[transition, bend left, /from to = {0}{0}{left}] (b) 
	;
\end{tikzpicture}
\label{subfig:32_abelian_end}
}
\caption{The five contransitive 3-state automata on a binary alphabet.}